\pdfoutput=1
\newif\ifpersonal
\RequirePackage[l2tabu, orthodox]{nag} 
\documentclass[12pt,a4paper,reqno]{amsart} 
\usepackage{amssymb,amsfonts,amsthm,amsmath,mathrsfs,mathtools}
\usepackage{microtype,fixltx2e}
\usepackage{enumerate,comment,braket,tikz-cd,xspace}
\usepackage[utf8]{inputenc}
\usepackage[T1]{fontenc}
\usepackage[hidelinks]{hyperref}
\usepackage[capitalize]{cleveref}
\usepackage[centering,vscale=0.7,hscale=0.7]{geometry}

\theoremstyle{plain}
\newtheorem{thm}{Theorem}[section]
\newtheorem*{thm*}{Theorem}
\newtheorem{lem}[thm]{Lemma}
\newtheorem{prop}[thm]{Proposition}
\newtheorem{cor}[thm]{Corollary}
\newtheorem*{cor*}{Corollary}
\theoremstyle{definition}
\newtheorem{defin}[thm]{Definition}
\theoremstyle{remark}

\newtheorem{rem}[thm]{Remark}
\numberwithin{equation}{section}

\ifpersonal
\newcommand*{\personal}[1]{\textcolor{blue}{(Personal: #1)}}
\newcommand*{\todo}[1]{\textcolor{red}{(Todo: #1)}}
\else
\newcommand*{\personal}[1]{\ignorespaces}
\newcommand*{\todo}[1]{\ignorespaces}
\fi

\newcommand{\Q}{\mathbb Q}
\newcommand{\R}{\mathbb R}
\newcommand{\Z}{\mathbb Z}

\newcommand{\fC}{\mathfrak C}

\newcommand{\fS}{\mathfrak S}
\newcommand{\fT}{\mathfrak T}
\newcommand{\fU}{\mathfrak U}

\newcommand{\fX}{\mathfrak X}
\newcommand{\fY}{\mathfrak Y}

\newcommand{\ff}{\mathfrak f}
\newcommand{\fm}{\mathfrak m}
\newcommand{\fs}{\mathfrak s}
\newcommand{\ft}{\mathfrak t}

\newcommand{\cD}{\mathcal D}

\newcommand{\cF}{\mathcal F}

\newcommand{\cO}{\mathcal O}

\newcommand{\bbD}{\mathbb D}
\newcommand{\bbG}{\mathbb G}

\newcommand{\oD}{\overline D}
\newcommand{\oM}{\overline M}

\newcommand{\hL}{\widehat L}

\newcommand{\IX}{I_\fX}

\newcommand{\SX}{S_\fX}

\newcommand{\SY}{S_\fY}

\newcommand{\oSX}{\overline{\SX}}
\newcommand{\oIX}{\overline{\IX}}

\newcommand{\fXe}{\fX_\eta}
\newcommand{\fXs}{\fX_s}
\newcommand{\ofX}{\overline{\fX}}

\newcommand{\fYe}{\fY_\eta}

\newcommand{\fXbs}{\fX_{\bar s}}

\newcommand{\QXe}{\Q_{\ell,\fX_\eta}}
\newcommand{\QXbs}{\Q_{\ell,\fXbs}}

\newcommand{\fCbs}{\fC_{\bar s}}

\newcommand{\QCe}{\Q_{\ell,\fC_\eta}}
\newcommand{\QCs}{\Q_{\ell,\fCbs}}

\newcommand{\bcMgn}{\overline{\mathcal M}_{g,n}}

\newcommand{\bMgnt}{\overline{M}^\mathrm{trop}_{g,n}}

\newcommand{\LanD}{\mathcal L_{an}^D}

\newcommand{\kc}{k^\circ}
\newcommand{\llb}{[\![}
\newcommand{\rrb}{]\!]}

\newcommand{\an}{^\mathrm{an}}

\newcommand{\et}{_\mathrm{\acute{e}t}}

\newcommand{\inv}{^{-1}}

\newcommand{\gn}{$n$-pointed genus $g$ }

\newcommand{\kanal}{$k$-analytic\xspace}

\renewcommand{\th}{^\mathrm{\tiny th}}

\newcommand{\trop}{^\mathrm{trop}}

\providecommand{\abs}[1]{\lvert#1\rvert}

\usetikzlibrary{decorations.markings} 
\tikzset{
  closed/.style = {decoration = {markings, mark = at position 0.5 with { \node[transform shape, xscale = .8, yscale=.4] {/}; } }, postaction = {decorate} },
  open/.style = {decoration = {markings, mark = at position 0.5 with { \node[transform shape, scale = .7] {$\circ$}; } }, postaction = {decorate} }
}

\DeclareMathOperator{\Coker}{Coker}

\DeclareMathOperator{\Div}{Div}

\DeclareMathOperator{\Ker}{Ker}

\DeclareMathOperator{\NE}{NE}

\DeclareMathOperator{\Spf}{Spf}

\DeclareMathOperator{\val}{val}

\begin{document}
\title[Tropicalization of the moduli space of stable maps]{Tropicalization of the moduli space of\\ stable maps}
\author{Tony Yue YU}
\address{Tony Yue YU, Institut de Mathématiques de Jussieu - Paris Rive Gauche, CNRS-UMR 7586, Case 7012, Université Paris Diderot - Paris 7, Bâtiment Sophie Germain 75205 Paris Cedex 13 France}
\email{yuyuetony@gmail.com}
\date{July 31, 2014 (revised on July 6, 2015)}
\subjclass[2010]{Primary 14T05; Secondary 14G22 14H15 03C98 03C10 32B20}
\keywords{Tropicalization, moduli space, stable map, continuity, polyhedrality, Berkovich space, balancing condition, vanishing cycle, quantifier elimination, rigid subanalytic set}

\begin{abstract}
Let $X$ be an algebraic variety and let $S$ be a tropical variety associated to $X$.
We study the tropicalization map from the moduli space of stable maps into $X$ to the moduli space of tropical curves in $S$.
We prove that it is a continuous map and that its image is compact and polyhedral.
Loosely speaking, when we deform algebraic curves in $X$, the associated tropical curves in $S$ deform continuously;
moreover, the locus of realizable tropical curves inside the space of all tropical curves is compact and polyhedral.
Our main tools are Berkovich spaces, formal models, balancing conditions, vanishing cycles and quantifier elimination for rigid subanalytic sets.
\end{abstract}

\maketitle

\personal{PERSONAL COMMENTS ARE SHOWN!!!}

\tableofcontents

\section{Introduction}\label{sec:intro_(tropicalization_moduli)}

Let $k$ be a complete discrete valuation field,
$X$ a toric variety over $k$ of dimension $d$,
and $C$ an algebraic curve embedded in $X$.
In tropical geometry (see\{ \cite{Mikhalkin_Tropical_ICM_2006,Itenberg_Tropical_2009,Gathmann_Tropical_2006}),
one associates to $C$ a piecewise-linear tropical curve $C\trop$ embedded in $\R^d$.
In this paper, we study this ``tropicalization procedure'' in families.

In fact, we do not restrict ourselves to toric target spaces.
We work with the framework of global tropicalization using Berkovich spaces (cf.\ \cite{Berkovich_Smooth_1999,Yu_Balancing_2013,Yu_Gromov_2014,Gubler_Skeletons_2014}).
This is not only more general, but also more natural from our viewpoint.

Roughly speaking, we prove the following results:
\begin{enumerate}[(i)]
\item The tropical curve $C\trop$ deforms continuously when we deform the algebraic curve $C$.
\item The locus of realizable tropical curves inside the space of all tropical curves of bounded degree is compact and polyhedral.
\end{enumerate}

More precisely, we fix a $k$-analytic space $X$ as the target space for curves.
We tropicalize $X$ by choosing a strictly semi-stable formal model $\fX$ of $X$ (see \cref{def:sss_formal_model}).
The associated tropical variety $\SX$ is homeomorphic to the dual intersection complex of the special fiber $\fXs$.
We call $\SX$ the \emph{Clemens polytope}.
As in the toric case, analytic curves in $X$ give rise to piecewise-linear tropical curves in $\SX$.

In order to bound the complexity of the analytic curves in $X$, we need to bound their degree with respect to a Kähler structure $\hL$ on $X$ (see \cref{sec:Kahler} and \cite{Yu_Gromov_2014,Kontsevich_Non-archimedean_2002}).
Similarly, we will define a notion of \emph{simple density} $\omega$ on $\SX$ induced by $\hL$ in order to bound the complexity of the tropical curves in $\SX$.

We have the following description of the space of tropical curves in $\SX$ with bounded degree (\cref{thm:tropical_compactness}).

\begin{thm*} 
Fix two non-negative integers $g,n$ and a positive real number $A$.
Let $M_{g,n}(\SX,A)$ denote the set of simple $n$-pointed genus $g$ parametrized tropical curves in $\SX$ whose degree with respect to $\omega$ is bounded by $A$.
Then $M_{g,n}(\SX,A)$ is naturally a compact topological space with a stratification whose open strata are open convex polyhedrons.
\end{thm*}

In view of applications, instead of only considering analytic curves embedded in $X$, we will consider stable maps into $X$ introduced by Maxim Kontsevich \cite{Kontsevich_Enumeration_1995}.
Let $T$ be a strictly \kanal space.
Assume we have a family over $T$ of \gn \kanal stable maps into $X$ with degree bounded by $A$.
We have the following set-theoretic tropicalization map
\[\tau_T\colon T\longrightarrow M_{g,n}(\SX,A)\]
which sends $n$-pointed genus $g$ \kanal stable maps to the associated $n$-pointed genus $g$ parametrized tropical curves.

\begin{thm*} 
The tropicalization map $\tau_T$ has the following properties:
\begin{enumerate}[(i)]
\item It is a continuous map (\cref{thm:continuity}).
\item Its image is polyhedral in $M_{g,n}(\SX,A)$, in the sense that the intersection with every open stratum of $M_{g,n}(\SX,A)$ is polyhedral (\cref{thm:relative_polyhedrality}).
\end{enumerate}
\end{thm*}

It is helpful to reformulate the theorem for the universal family of stable maps.
Let $\bcMgn(X,A)$ denote the moduli stack of \gn \kanal stable maps into $X$ with degree bounded by $A$.
It is a compact \kanal stack by \cite{Yu_Gromov_2014}.
The map $\tau_T$ above extends to a tropicalization map
\[\tau_M\colon \bcMgn(X,A)\longrightarrow M_{g,n}(\SX,A).\]

The image of the map $\tau_M$ consists of so-called \emph{realizable tropical curves}.
We denote it by $\bMgnt(X,A)$.
It is of much interest in tropical geometry to characterize realizable tropical curves (cf.\ \cite{Mikhalkin_Tropical_ICM_2006,Nishinou_Toric_2006,Speyer_Uniformizing_2007,Nishinou_Correspondence_2009,Tyomkin_Tropical_2012,Cheung_Faithful_realizability_2014,Nishinou_Yu_Realization_2015}).
We have the following corollary concerning $\tau_M$ and the locus of realizable tropical curves.

\begin{cor*}
\begin{enumerate}[(i)]
\item The tropicalization map $\tau_M$ is a continuous map.
\item The locus $\bMgnt(X,A)$ is compact and polyhedral.
\end{enumerate}
\end{cor*}

\medskip

\paragraph{\textbf{Discussions and related works}}
One motivation of our work stems from the speculations by Kontsevich and Soibelman in \cite[\S 3.3]{Kontsevich_Homological_2001} and the works by Gross, Siebert, Hacking and Keel \cite{Gross_Real_Affine_2011,Gross_Tropical_2011,Gross_Mirror_Log_2011}.
We will apply our results to the study of non-archimedean enumerative geometry (see \cite{Yu_Enumeration_cylinders_2015}).
Another motivation is a question asked by Ilia Itenberg during a talk given by the author at Jussieu, Paris.

Moduli spaces of tropical curves in classical contexts were studied by Mikhalkin, Nishinou, Siebert, Gathmann, Markwig, Kerber, Kozlov, Caporaso, Viviani, Brannetti, Melo, Chan, Yu and others in \cite{Mikhalkin_Tropical_ICM_2006,Mikhalkin_Moduli_spaces_2007,Nishinou_Toric_2006,Gathmann_Kontsevich_formula_2008,Gathmann_Tropical_fans_2009,Kozlov_The_topology_2009,Caporaso_Torelli_2010,Kozlov_Moduli_spaces_2011,Brannetti_On_the_tropical_2011,Caporaso_Algebraic_and_tropical_2013,Chan_Combinatorics_2012,Chan_Tropical_Teichmuller_2013,Gross_Logarithmic_2013,Yu_Number_2013}.

Very interesting geometry concerning the tropicalization of the moduli space of stable curves is studied in detail by Abramovich, Caporaso and Payne \cite{Abramovich_Tropicalization_2012}.
More generally, we expect to have explicit descriptions for the tropicalization of the moduli space of stable maps into toric varieties.
There are related developments by Cavalieri, Markwig, Ranganathan, Ascher, Molcho, Chen, Satriano and A.\ Gross \cite{Cavalieri_Tropical_compactification_2014,Cavalieri_Tropicalizing_the_space_2014,Ascher_Logarithmic_stable_2014,Chen_Chow_quotients_2013,Gross_Correspondence_theorems_2014,Ranganathan_Moduli_of_rational_2015}.

The model-theoretic technique involved in the proof of polyhedrality is inspired by the works of Ducros \cite{Ducros_Espaces_de_Berkovich_polytopes_2012} and Martin \cite{Martin_Constructibilite_2013}, and is based on the theory of rigid subanalytic sets developed by Lipshitz and Robinson \cite{Lipshitz_Rigid_1993,Lipshitz_Model_completeness_2000}.

We regret a certain asymmetry in our results.
On the analytic side, we consider stable maps; while on the tropical side, we consider parametrized tropical curves which are locally embedded.
One can define the notion of tropical stable maps and study their moduli space.
We conjecture that the tropicalization map from the space of analytic stable maps to the space of tropical stable maps is also continuous and has polyhedral image.

\medskip

\paragraph{\textbf{Outline of the paper}}

In Section \ref{sec:basic_settings_(tropicalization_moduli)}, we review the basic settings of global tropicalization.
Given a $k$-analytic space $X$, the tropicalization $\SX$ of $X$ depends on the choice of a formal model $\fX$ of $X$.
We work with strictly semi-stable formal models for simplicity.

In \cref{sec:parametrized_tropical_curves}, we define the notion of parametrized tropical curve in our context.
We explain how analytic curves in $X$ give rise to tropical curves in $\SX$.

The tropical curves in $\SX$ that arise from analytic curves satisfy a distinguished geometrical property, called the \emph{balancing condition}.
It is a generalization of the classical balancing condition (cf.\ \cite{Mikhalkin_Enumerative_2005,Nishinou_Toric_2006,Speyer_Uniformizing_2007,Baker_Nonarchimedean_2011}).
The balancing condition in the global setting was first studied in \cite{Yu_Balancing_2013} using $k$-analytic cohomological arguments.
In Section \ref{sec:balancing_conditions}, we give a different proof which is useful for the purpose of this paper.
The main observation is that the tropical weight vectors can be read out directly from certain intersection numbers via the functor of vanishing cycles (\cref{lem:tropical_weight_vector_as_intersection_number}).

In \cref{sec:Kahler}, we introduce a combinatorial notion of simple density on $\SX$.
We define the degree of a tropical curve with respect to a simple density.
We use it to give a lower bound of the degree of an analytic curve in $X$ with respect to a non-archimedean Kähler structure on $X$.

In \cref{sec:space_of_tropical_curves}, we study the space of tropical curves in $\SX$ with bounded degree.
We use some combinatorial arguments from our previous work \cite{Yu_Number_2013}.

In \cref{sec:stack_of_stable_maps}, we review the moduli stack of \kanal stable maps constructed in \cite{Yu_Gromov_2014}.

In \cref{sec:continuity_of_tropicalization}, we prove the continuity of the tropicalization map from the \kanal moduli space of \cref{sec:stack_of_stable_maps} to the tropical moduli space of \cref{sec:space_of_tropical_curves}.
The proof makes use of the balacing conditions in \cref{sec:balancing_conditions} and the formal models of families of \kanal stable maps developed in \cite{Yu_Gromov_2014}.

In \cref{sec:polyhedrality}, combining the continuity theorem of \cref{sec:continuity_of_tropicalization} with the quantifier elimination theorem from the model theory of rigid subanalytic sets \cite{Lipshitz_Rigid_1993}, we prove the polyhedrality of the locus of realizable tropical curves inside the space of all tropical curves.

\medskip

\paragraph{\textbf{Acknowledgments}}
I am very grateful to Maxim Kontsevich and Antoine Chambert-Loir for inspirations and support.
Special thanks to Antoine Ducros from whom I learned model theory and its applications to tropical geometry.
I appreciate valuable discussions with Vladimir Berkovich, Pierrick Bousseau, Ilia Itenberg, François Loeser, Florent Martin, Johannes Nicaise, Sam Payne and Michael Temkin.
Comments given by the referees helped greatly improve the paper.

\section{Basic settings of global tropicalization}\label{sec:basic_settings_(tropicalization_moduli)}
In this section, we review the basic settings of global tropicalization.
We refer to \cite[\S 2]{Yu_Gromov_2014} and \cite[\S 3]{Boucksom_Singular_2011} for more details
(see also \cite{Kontsevich_Non-archimedean_2002,Kempf_Toroidal_1973,Gubler_Skeletons_2014,Yu_Balancing_2013}).

Let $k$ be a complete discrete valuation field.
Denote by $k^\circ$ the ring of integers, $k^{\circ\circ}$ the maximal ideal, and $\tilde k$ the residue field.

For $n\geq 0$, $0\leq d\leq n$ and $a\in k^{\circ\circ}\setminus 0$, put
\begin{equation}\label{eq:standard_formal_scheme_(tropicalization_moduli)}
\fS(n,d,a) = \Spf \left(k^\circ\{T_0,\dots,T_n, T\inv_{d+1},\dots,T\inv_n\}/(T_0\cdots T_d-a)\right) ,
\end{equation}
where $\Spf$ denotes the formal spectrum.

\begin{defin}[\cite{Berkovich_Smooth_1999}]
A formal scheme $\fX$ is said to be \emph{finitely presented} over $k^\circ$ if it is a finite union of open affine subschemes of the form \[\Spf\big(k^\circ\{T_0,\dots,T_n\}/(f_1,\dots,f_m)\big) .\]
\end{defin}

Let $\fX$ be a formal scheme finitely presented over $k^\circ$.
One can define its generic fiber $\fXe$ and its special fiber $\fXs$ following \cite{Berkovich_Vanishing_1994}.
Its generic fiber $\fXe$ has the structure of a compact strictly \kanal space in the sense of Berkovich \cite{Berkovich_Spectral_1990,Berkovich_Etale_1993}, and its special fiber $\fXs$ is a scheme of finite type over the residue field $\tilde k$.
We denote by $\pi\colon\fXe\to\fXs$ the reduction map from the generic fiber to the special fiber (cf.\ \cite[\S 1]{Berkovich_Vanishing_1994}).

\begin{defin}\label{def:formal_model}
Let $X$ be a \kanal space. A \emph{(finitely presented) formal model} of $X$ is a formal scheme $\fX$ finitely presented over $k^\circ$ together with an isomorphism between the generic fiber $\fXe$ and the \kanal space $X$.
\end{defin}

\begin{defin}\label{def:sss_formal_model}
Let $\fX$ be a formal scheme finitely presented over $k^\circ$. 
The formal scheme $\fX$ is said to be \emph{strictly semi-stable} if
\begin{enumerate}[(i)]
\item Every point $x$ of $\fX$ has an open affine neighborhood $\fU$ such that the structure morphism $\fU\rightarrow\Spf k^\circ$ factors through an étale morphism $\phi\colon \fU\rightarrow\mathfrak S(n,d,\varpi)$ for some $0\leq d\leq n$ and a uniformizer $\varpi$ of $k$.
\item All the intersections of the irreducible components of the special fiber $\fXs$ are either empty or geometrically irreducible.
\end{enumerate}
\end{defin}

Let $X$ be a \kanal space and let $\fX$ be a strictly semi-stable formal model\footnote{If the residue field $\tilde k$ has characteristic zero and if $X$ is compact quasi-smooth and strictly \kanal, then strictly semi-stable formal models exist up to passing to a finite extension of $k$ (cf.\ \cite{Temkin_Desingularization_2008,Kempf_Toroidal_1973}).} of $X$.
Let $\Set{D_i|i\in\IX}$ denote the set of the irreducible components of the special fiber $\fXs$.
For every non-empty subset $I\subset\IX$,
put $D_I=\bigcap_{i\in I} D_i$ and 
\begin{equation}\label{eq:J_(tropicalization_moduli)}
J_I=\Set{j\in\IX | D_{I\cup\{j\}}\neq\emptyset}.
\end{equation}
Condition (i) of \cref{def:sss_formal_model} implies that all the strata $D_I$ are smooth over the residue field $\tilde k$.

The \emph{Clemens polytope} $\SX$ is by definition the simplicial subcomplex of the simplex $\Delta^{\IX}$ such that for every non-empty subset $I\subset\IX$, the simplex $\Delta^I$ is a face of $\SX$ if and only if the stratum $D_I$ is non-empty.
As a special case of \cite{Berkovich_Smooth_1999}, one can construct a canonical inclusion map $\theta\colon \SX\hookrightarrow\fX_\eta$ and a canonical strong deformation retraction from $\fX_\eta$ to $\SX$.
For simplicity, we only explain the construction of the retraction map $\tau\colon\fXe\rightarrow\SX$, i.e.\ the final moment of the strong deformation retraction.

Let $\Div_0(\fX)_\R$ denote the vector space of vertical Cartier $\R$-divisors on $\fX$.
It is of dimension the cardinality of $\IX$.
An effective vertical divisor $D$ on $\fX$ is locally given by a function $u$ up to multiplication by invertible functions.
So $\val(u(x))$ defines a continuous function on $\fXe$ which we denote by $\varphi^0_D$.
By linearity, $\varphi^0_D$ makes sense for any divisor $D$ in $\Div_0(\fX)_\R$.
Let $\tau\colon\fXe\rightarrow\Div_0(\fX)_\R^*$ be the evaluation map defined by $\langle\tau(x),D\rangle=\varphi_D^0(x)$ for any $x\in\fXe$, $D\in\Div_0(\fX)_\R$.
The image of $\tau$ can be naturally identified with the Clemens polytope $\SX$.
The identification gives us a canonical embedding
\begin{equation}
\SX\subset \Div_0(\fX)^*_\R\simeq \R^{\IX}.
\end{equation}
We will always regard the Clemens polytope $\SX$ as embedded in $\Div_0(\fX)^*_\R$.

\begin{rem}\label{rem:explicit}
Let $\fS$ denote the standard formal scheme $\fS(n,d,a)$ in \eqref{eq:standard_formal_scheme_(tropicalization_moduli)}.
The retraction map $\tau_\fS\colon\fS_\eta\to S_\fS\subset\R^{d+1}$ can be written explicitly as follows
\begin{align*}
\fS_\eta&\longrightarrow\R^{d+1}\\
x&\longmapsto\big(\val T_0(x),\dots,\val T_d(x)\big).
\end{align*}
The image is the $d$-dimensional simplex in $\R^{d+1}_{\geq 0}$ given by the equation $\sum x_i=\val(a)$.
We remark that in the general case, the retraction map $\tau\colon\fXe\to\SX\subset\R^{\IX}$ is locally of the form above.
\end{rem}

The retraction from the $k$-analytic space $X$ to the Clemens polytope $\SX$ with respect to the formal model $\fX$ is functorial in the following sense.

\begin{prop}[cf.\ {\cite[\S 2]{Yu_Gromov_2014}}]\label{prop:functoriality}
Let $f\colon X\to Y$ be a morphism of $k$-analytic spaces.
Let $\fX$ and $\fY$ be strictly semi-stable formal models of $X$ and $Y$ respectively such that the morphism $f$ extends to a morphism $\ff\colon\fX\to\fY$ of formal schemes.
Let $\tau_\fX\colon\fXe\to\SX$ and $\tau_\fY\colon\fYe\to S_\fY$ denote the retraction maps.
Then there exists a continuous map $S_\ff\colon\SX\to\SY$, which is affine on every simplicial face of $\SX$, such that the diagram
\begin{equation}
\begin{tikzcd}
\fXe \arrow{r}{\ff_\eta} \arrow{d}{\tau_\fX} & \fYe \arrow{d}{\tau_\fY} \\
\SX \arrow{r}{S_{\ff}} & S_\fY
\end{tikzcd}
\end{equation}
commutes.
\end{prop}

\section{Parametrized tropical curves and parametrized tropicalization} \label{sec:parametrized_tropical_curves}

In this section, we introduce the notions of parametrized tropical curves, combinatorial types and degenerations of combinatorial types.
After that, we explain how analytic curves give rise to tropical curves.
We use the settings of \cref{sec:basic_settings_(tropicalization_moduli)}.

\begin{defin}
Let $\Gamma$ be a finite undirected graph.
We denote by $V(\Gamma)$ the set of vertices and by $E(\Gamma)$ the set of edges.
For a vertex $v$ of $\Gamma$, the degree $\deg(v)$ denotes the number of edges connected to $v$.
For two vertices $u,v$ of $\Gamma$, we denote by $E(u,v)$ the set of edges connecting $u$ and $v$.
For a vertex $v$ and an edge $e$ of $\Gamma$, we denote $v\in e$ or $e\ni v$ if $v$ is an endpoint of $e$.
A \emph{flag} of $\Gamma$ is a pair $(v,e)$ consisting of a vertex $v$ and an edge $e$ connected to $v$.
\end{defin}

\begin{defin}\label{def:parametrized_tropical_curve}
An \emph{$n$-pointed parametrized tropical curve} $(\Gamma,(\gamma_i),h)$ in the Clemens polytope $\SX$ consists of the following data:
\begin{enumerate}[(i)]
\item A connected finite graph $\Gamma$ without self-loops.
\item A continuous map $h$ from the topological realization of $\Gamma$ to $\SX$ such that every edge of $\Gamma$ embeds as an affine segment with rational slope in a face of $\SX$.
\item \label{item:tropical_curve:weight} Every flag $(v,e)$ of $\Gamma$ is equipped with a \emph{tropical weight vector} $w_{(v,e)}\in\Z^{\IX}\setminus 0$, parallel to the direction of $h(e)$ pointing away from $h(v)$.
For every edge $e$ of $\Gamma$ and its two endpoints $u$ and $v$, we require that $w_{(u,e)} + w_{(v,e)} = 0$.
\item Every vertex $v$ of $\Gamma$ is equipped with a non-negative integer $g(v)$, called the \emph{genus} of the vertex $v$.
\item A sequence of vertices $\gamma_1,\dots,\gamma_n$ of $\Gamma$, called \emph{marked points}.
They are not required to be different from each other.
For each vertex $v$ of $\Gamma$, we denote by $n(v)$ the number of marked points at $v$.
\end{enumerate}
We denote a $0$-pointed parametrized tropical curve simply by $(\Gamma,h)$.
\end{defin}

\begin{defin}\label{def:type_of_vertices}
Let $(\Gamma,(\gamma_i),h)$ be an $n$-pointed parametrized tropical curve in $\SX$.
We define its \emph{genus} to be the sum
\[b_1(\Gamma)+\sum_{v\in V(\Gamma)} g(v),\]
where $b_1(\cdot)$ denotes the first Betti number.
For every vertex $v$ of $\Gamma$, we define the \emph{sum of weight vectors} around $v$ to be $\sigma_v\coloneqq\sum_{e\ni v} w_{(v,e)}$, summing over all edges connected to $v$.
A vertex $v$ of $\Gamma$ is said to be of \emph{type A} if $\sigma_v$ is nonzero.
Otherwise it is said to be of \emph{type B}.
An $n$-pointed parametrized tropical curve $(\Gamma,(\gamma_i),h)$ is said to be \emph{simple} if every vertex $v$ with $\deg(v)=2$ and $g(v)=n(v)=0$ is of type A.
\end{defin}

\begin{defin}\label{def:simplification}
Given an $n$-pointed parametrized tropical curve $(\Gamma,(\gamma_i),h)$ in $\SX$, one can obtain a unique simple $n$-pointed parametrized tropical curve as follows:
for every vertex $v$ of $\Gamma$ of type B such that $\deg(v)=2$ and $g(v)=n(v)=0$, we remove the vertex $v$, replace the two edges connected to $v$ by a single edge, and set the tropical weight vectors accordingly.
We call this construction \emph{simplification}.
\end{defin}

\begin{defin}\label{def:combinatorial_type}
An \emph{$n$-pointed combinatorial type} $(\Gamma,(\gamma_i))$ in $\SX$ consists of the following data:
\begin{enumerate}[(i)]
\item A connected finite graph $\Gamma$ without self-loops.
\item Every vertex $v$ of $\Gamma$ is equipped with a non-empty subset $I_v\subset\IX$.
\item Every flag $(v,e)$ of $\Gamma$ is equipped with a \emph{tropical weight vector} $w_{(v,e)}\in\Z^{\IX}\setminus 0$.
For every edge $e$ of $\Gamma$ and its two endpoints $u$ and $v$, we require that $w_{(u,e)} + w_{(v,e)} = 0$.
\item Every vertex $v$ of $\Gamma$ is equipped with a non-negative integer $g(v)$, called the \emph{genus} of the vertex $v$.
\item A sequence of vertices $\gamma_1,\dots,\gamma_n$ of $\Gamma$, called \emph{marked points}.
They are not required to be different from each other.
For each vertex $v$ of $\Gamma$, we denote by $n(v)$ the number of marked points at $v$.
\end{enumerate}
\end{defin}

Definitions \ref{def:type_of_vertices} and \ref{def:simplification} carry over to combinatorial types.
Given an $n$-pointed parametrized tropical curve $(\Gamma,(\gamma_i),h)$ in $\SX$, we can associate to it an $n$-pointed combinatorial type $(\Gamma,(\gamma_i))$ by letting $I_v\subset \IX$ be the subset such that the vertex $v$ sits in the relative interior of the face $\Delta^{I_v}$ for every vertex $v$ of $\Gamma$.

\begin{defin}
An $n$-pointed combinatorial type in $\SX$ is said to be \emph{good} if it comes from an $n$-pointed parametrized tropical curve in $\SX$.
\end{defin}

\begin{defin}\label{def:degeneration_combinatorial_type}
An $n$-pointed combinatorial type $(\Gamma,(\gamma_i))$ in $\SX$ is said to be a \emph{degeneration} of an $n$-pointed combinatorial type $(\Gamma',(\gamma'_i))$ in $\SX$ if there exists a surjective map $\phi\colon V(\Gamma')\to V(\Gamma)$ satisfying the following conditions:
\begin{enumerate}[(i)]
\item \label{item:degeneration:edge} For any two vertices $u',v'$ of $\Gamma'$ such that $\phi(u')\neq\phi(v')$, there exists a bijection $\phi\colon E(u',v') \to E(\phi(u'),\phi(v'))$ such that for every $e'\in E(u',v')$, we have $w_{(u',e')}=w_{(\phi(u'),\phi(e'))}$ and $w_{(v',e')}=w_{(\phi(v'),\phi(e'))}$.
Moreover, we require that every edge $e$ of $\Gamma$ is of the form $\phi(e')$ for some edge $e'$ of $\Gamma'$.
\item \label{item:degeneration:face} For every vertex $v'$ of $\Gamma'$, we have $I_{\phi(v')}\subset I_{v'}$.
\item \label{item:degeneration:genus} For every vertex $v$ of $\Gamma$, let $\Gamma'_{v}$ denote the full subgraph of $\Gamma'$ generated by the preimage $\phi\inv(v)$.
We require that $\Gamma'_{v}$ is connected and that
\[g(v)=b_1(\Gamma'_{v})+\sum_{v'\in\phi\inv(v)} g(v'),\]
where $b_1(\cdot)$ denotes the first Betti number.
\item \label{item:degeneration:points} We have $\phi(\gamma'_i)=\gamma_i$ for $1\le i\le n$.
\end{enumerate}
\end{defin}

\bigskip

Now given a connected compact quasi-smooth strictly \kanal curve $C$, $n$ marked points $s_i\in C$ and a morphism $f\colon C\to X$, we can obtain a simple $n$-pointed parametrized tropical curve $(\Gamma,(\gamma_i),h)$ in $\SX$ by the following steps.

\medskip
\noindent\emph{Step 1.}
Up to passing to a finite separable extension of $k$, there exists a strictly semi-stable formal model $\fC$ of $C$ and a morphism of formal schemes $\ff\colon \fC\rightarrow\fX$, such that $\ff_\eta\simeq f$.
This a consequence of \cite[Proposition 5.1]{Yu_Balancing_2013} using the correspondence between finite sets of type II points in $C$ and semi-stable reductions of $C$ (see \cite{Ducros_Structure_2012,Baker_Structure_2013}).
When $C$ is proper, this is also a special case of \cite[Theorem 1.5]{Yu_Gromov_2014}.
Let $S_\fC$ be the Clemens polytope for $\fC$ and $\tau_\fC\colon C\to S_\fC$ the retraction map.
By \cref{prop:functoriality}, we obtain a map $S_\ff\colon S_\fC\to\SX$.
We put $\Gamma_0\coloneqq S_\fC$ and $h_0\coloneqq S_\ff$.
For every vertex $v$ of $\Gamma_0$, let $\fC_s^v$ denote the corresponding irreducible component of $\fC_s$.
We set $g(v)$ to be the genus of $\fC_s^v$.
We add to $\Gamma_0$ the marked points $\tau_\fC(s_i)$, creating new vertices if necessary.

\medskip
\noindent\emph{Step 2.}
For every flag $(v,e)$ of $\Gamma_0$, we define a weight vector $w_{(v,e)}\in\Z^{\IX}$ as follows (see also \cite[\S 5]{Yu_Balancing_2013}).
Let $e^\circ$ denote the relative interior of the edge $e$.
The inverse image $\tau_\fC\inv(e^\circ)$ is an open annulus in $C$, which we denote by $A$.
Fix $i\in\IX$ and let $p_i\colon\R^{\IX}\rightarrow\R$ be the projection to the $i^\text{th}$ coordinate.
By \cref{rem:explicit}, the map $p_i\circ\tau\circ f|_{A}$ is given by the valuation of a certain invertible function $f_i$ on $A$.
Let $z$ be a coordinate on the annulus $A$ such that the annulus $A$ is given by $c_1<|z|<c_2$ and that $c_1$ corresponds to the vertex $v$.
We write $f_i=\sum_{m\in\Z}f_{i,m} z^m$, where $f_{i,m}\in k$.
Since $f_i$ is invertible on $A$, there exists $m_i\in\Z$ such that $|f_{i,m_i}|r^{m_i}>|f_{i,m}|r^m$ for all $m\neq m_i$ and $c_1<r<c_2$.
We set the $i^\text{th}$ component of the weight vector $w_{(v,e)}$ to be $m_i$.
By construction, the weight vector $w_{(v,e)}$ is parallel to the direction of $h_0(e)$ pointing away from $h_0(v)$.
It is zero if and only if the edge $e$ is mapped to a point by $h_0$.

\medskip
\noindent\emph{Step 3.}
For every connected subgraph $\Gamma'$ of $\Gamma_0$ that is mapped to a point by $h_0$, we contract it to a vertex $v'$ and set
\[ g(v') = b_1(\Gamma') + \sum_{v\in \Gamma'} g(v).\]
We denote the resulting $n$-pointed parametrized tropical curve by $(\Gamma,(\gamma_i),h)$.
After that, we replace $(\Gamma,(\gamma_i),h)$ by its simplification (see \cref{def:simplification}).
We note that Step 3 removes the dependence of our construction on the choices of the finite extension of $k$ and the formal model $\fC$ of $C$.

\section{Vanishing cycles and balancing conditions}\label{sec:balancing_conditions}

We use the settings of \cref{sec:basic_settings_(tropicalization_moduli)}.

Let $C$ be a connected proper smooth strictly $k$-analytic curve and let $f\colon C\to X$ be a morphism.
Let $(\Gamma,h)$ be the associated parametrized tropical curve in $\SX$ constructed in \cref{sec:parametrized_tropical_curves}.

The local shape of the tropical curve $(\Gamma,h)$ sitting inside $\SX$ satisfies a distinguished geometrical property, which we refer to as the \emph{balancing condition}.
It is studied in \cite[Theorem 1.1]{Yu_Balancing_2013}; see also \cite[Proposition 1.15]{Gross_Logarithmic_2013} for a related statement.
Here we give a different proof using semi-stable reduction of \kanal curves, which will be useful for later sections.

Let $v$ be a vertex of $\Gamma$.
Assume that $h(v)$ sits in the relative interior of the face $\Delta^{I_v}$ of $\SX$ corresponding to a subset $I_v\subset I_\fX$.
Let $D_{I_v}$ denote the closed stratum in the special fiber $\fXs$ corresponding to the face $\Delta^{I_v}$.

Let $Z_1(D_{I_v})$ denote the group of one-dimensional algebraic cycles in $D_{I_v}$ with integer coefficients.
Let $Z_1^+(D_{I_v})$ denote the submonoid consisting of effective cycles, i.e.\ cycles with non-negative integer coefficients.
Let $\alpha$ be the map
\begin{align}
\begin{split}\label{eq:map_alpha}
\alpha\colon Z_1^+(D_{I_v})&\longrightarrow \Z^{\IX}\\
Z \qquad &\longmapsto \Big( Z\cdot\cO(D_i)|_{D_{I_v}},\ i\in \IX\,\Big),
\end{split}
\end{align}
which sends a one-dimensional cycle $Z$ to its intersection numbers with the restrictions $\cO(D_i)|_{D_{I_v}}$ for every $i\in \IX$.

\begin{thm}\label{thm:balancing_conditions}
The sum $\sigma_v$ of weight vectors around the vertex $v$ lies in the image of the map $\alpha$ defined in \eqref{eq:map_alpha}.
\end{thm}

Now we explain the proof of Theorem \ref{thm:balancing_conditions}.

Let $k^s$ be a separable closure of $k$, $\widehat{k^s}$ its completion, and $\widetilde{k^s}$ its residue field.
We fix a prime number $\ell$ different from the characteristic of the residue field $\widetilde k$.
Let $R\Psi$ and $R\Phi$ denote the derived functors of nearby cycles and of vanishing cycles respectively (cf.\ \cite[\S 3]{Yu_Balancing_2013} and \cite{Berkovich_Vanishing_1994}\footnote{We use the terminology in \cite{Yu_Balancing_2013}, which is different from \cite{Berkovich_Vanishing_1994,Berkovich_Vanishing_II_1996}.}).
We have the following exact triangle
\begin{equation}\label{eq:exact_triangle}
\QXbs\rightarrow R\Psi\QXe \rightarrow R\Phi\QXe \xrightarrow{+1},
\end{equation}
where $\Q_{\ell,-}$ denotes the constant sheaf with values in $\Q_\ell$, and $\fXbs = \fXs\times \widetilde{k^s}$.

Let $\oD_{I_v}=D_{I_v}\times\widetilde{k^s}$ and let $j\colon \oD_{I_v}\hookrightarrow\fXbs$ denote the closed immersion.
We apply $j^*$ to \eqref{eq:exact_triangle} and take global sections, we obtain a long exact sequence
\[\cdots\to R^1\Gamma\big(j^* R\Psi \QXe\big)\xrightarrow{\beta^*} R^1\Gamma\big(j^* R\Phi\QXe\big)\xrightarrow{\alpha^*} H^2\et\big(\oD_{I_v},\Q_\ell)\to\cdots,\]
where $\alpha^*$ is the boundary map.

Put $J=J_{I_v}=\Set{j\in\IX|D_{I_v\cup\{j\}}\neq\emptyset}$.
Recall that we have the following calculation of the sheaf of vanishing cycles for a strictly semi-stable formal scheme $\fX$ over $k^\circ$.

\begin{thm}[cf.\ {\cite[Corollary 3.2]{Yu_Balancing_2013}}, \cite{Rapoport_Lokale_1982}, \cite{Illusie_Semistable_2004}]\label{thm:vanishing_cycles}
We have an isomorphism
\[R^1\Gamma\left(j^* R\Phi\QXe\right)\simeq \Coker\big(\Q_\ell\xrightarrow{\Delta}\Q_\ell^J\big)(-1) ,\]
where $\Delta$ denotes the diagonal map and the symbol $(-1)$ denotes the Tate twist.
Moreover, the boundary map
\[\alpha^*\colon \Coker\big(\Q_\ell\xrightarrow{\Delta}\Q_\ell^J\big)(-1)\longrightarrow R^2\Gamma\big(j^*\QXbs\big)\simeq H^2\et\big(\oD_{I_v},\Q_\ell\big)\]
is induced by the cycle class map in étale cohomology.
\end{thm}

Let $\fC$, $\ff\colon\fC\to\fX$, $(\Gamma_0,h_0)$ be as in \cref{sec:parametrized_tropical_curves}.
We observe that in order to prove \cref{thm:balancing_conditions} for the vertex $v$ of $(\Gamma,h)$, it suffices to prove that for every vertex $v_0$ of $(\Gamma_0,h_0)$ that maps to $v$, the sum $\sigma_{v_0}$ of weight vectors around $v_0$ lies in the image of the map $\alpha$ defined in \eqref{eq:map_alpha}.

Let $v_0$ be a vertex of $\Gamma_0$ and let $\fC_s^{v_0}$ denote corresponding irreducible component of the special fiber $\fC_s$.
The assumption that $h_0(v_0)$ lies in the relative interior of the face $\Delta^{I_v}$ implies that the image of $\fC_s^{v_0}$ under the map $\ff_s$ is contained in $D_{I_v}$.

Let $\fCbs = \fC_s\times\widetilde{k^s}$, $\fC^{v_0}_{\bar s} = \fC^{v_0}_s\times\widetilde{k^s}$.
Denote by $j_C\colon \fC^{v_0}_{\bar s}\rightarrow \fCbs$ the closed immersion.
Put $\ff_{\bar s}\coloneqq \ff_s\times\widetilde{k^s}\colon\fCbs\to\fXbs$.
The properness of $\fC_s^{v_0}$ implies an isomorphism $H\et^2\big(j_C^*\QCs\big)\simeq\Q_\ell(-1)$.

\begin{lem}
We have the following commutative diagram:
\begin{equation}
\begin{tikzcd}[column sep=small]\label{eq:cd_vanishing1}
R^1\Gamma\big(j^*R\Psi\QXe\big) \arrow{r}{\beta^*} \arrow{d}{\ff^*_\Psi} 
& R^1\Gamma\big(j^*R\Phi\QXe\big) \arrow{r}{\alpha^*} \arrow{d}{\ff^*_\Phi}
& H\et^2\big(j^*\QXbs\big) \arrow{d}{\ff^*_{s}} & \\
R^1\Gamma\big(j_C^*R\Psi\QCe\big) \arrow{r}{\beta_C^*} 
& R^1\Gamma\big(j_C^*R\Phi\QCe\big) \arrow{r}{\alpha_C^*}
& H\et^2\big(j_C^*\QCs\big) \arrow{r}{\simeq} & \Q_\ell(-1).
\end{tikzcd}
\end{equation}
\end{lem}
\begin{proof}
For any étale sheaf $\cF$ on $\fCbs$, we have the adjunction morphism
\begin{equation}\label{eq:ajunction}
\cF\rightarrow j_{C*} j_C^* \cF .
\end{equation}
Applying the derived pushforward functor $R\ff_{\bar s*}$ to both sides of \eqref{eq:ajunction}, we obtain a morphism
\[R\ff_{\bar s*}\cF\rightarrow R\ff_{\bar s*} j_{C*} j_C^*\cF .\]
Since the image of $\fC^{v_0}_s$ under the map $\ff_s$ is contained in $D_{I_v}$, the sheaf $R\ff_{\bar s*}j_{C*} j_C^*\cF$ is supported on $\oD_{I_v}\subset \fXbs$. Therefore, we obtain a morphism
\begin{equation}\label{eq:balancing1}
j_* j^* R\ff_{\bar s*}\cF\rightarrow R\ff_{\bar s*}j_{C*} j_C^*\cF .
\end{equation}
Moreover, by \cite[Corollary 4.5(ii)]{Berkovich_Vanishing_1994}, we obtain a morphism 
\[R\Psi\QXe\rightarrow R\Psi Rf_*\QCe \simeq R\ff_{\bar s*} R\Psi\QCe .\]
Applying $j_* j^*$, we obtain a morphism
\begin{equation}\label{eq:balancing2}
j_*j^* R\Psi\QXe\rightarrow j_* j^* R\ff_{\bar s*} R\Psi\QCe.
\end{equation}
Substituting $\cF$ by $R\Psi\QCe$ in \eqref{eq:balancing1}, we obtain a morphism
\[j_* j^* R\ff_{\bar s*} R\Psi\QCe \rightarrow R\ff_{\bar s*} j_{C*} j_C^* R\Psi\QCe .\]
Combining with \eqref{eq:balancing2} and taking global sections, we obtain a map
\[\ff^*_\Psi\colon R^1\Gamma\big(j^* R\Psi\QXe\big)\rightarrow R^1\Gamma\big(j_C^* R\Psi\QCe\big).\]
Similarly, we have maps
\[\ff^*_\Phi\colon R^1\Gamma\big(j^* R\Phi\QXe\big)\rightarrow R^1\Gamma\big(j_C^* R\Phi\QCe\big),\]
and
\[\ff^*_s\colon H^2\et\big(j^*\QXbs\big)\rightarrow H^2\et\big(j_C^*\QCs\big).\]
Now the commutativity of \eqref{eq:cd_vanishing1} follows from the functoriality of nearby cycles and vanishing cycles.
\end{proof}

By Theorem \ref{thm:vanishing_cycles}, we have an isomorphism
\begin{equation}\label{eq:vanishing_cycles_at_DIv}
R^1\Gamma\big(j^*R\Phi\QXe\big)\simeq \Coker\big(\Q_\ell\xrightarrow{\Delta}\Q_\ell^J\big)(-1) .
\end{equation}

For each edge $e$ connected to $v_0$, the weight vector $w_{(v_0,e)}$ lives in $\Ker\big(\Z^J\xrightarrow{\Sigma}\Z\big)$.
So it induces a linear map by duality
\[w_{(v_0,e)}^*\colon R^1\Gamma\big(j^*R\Phi\QXe\big) \rightarrow \Q_\ell(-1).\]

Let $p_c\in\fC_s$ be the point corresponding to the edge $e$.
Let $j_{p_c}\colon p_c\times\widetilde{k^s}\hookrightarrow {\fCbs^{v_0}}$ denote the inclusion map.
By Theorem \ref{thm:vanishing_cycles} again, we have an isomorphism
\[R^1\Gamma\big(j_{p_c}^*R\Phi\QCe\big)\simeq \Coker\big(\Q_\ell\xrightarrow{\Delta} \Q_\ell\oplus \Q_\ell\big)(-1).\]

Let $s$ be the projection map
\begin{align*}
R^1\Gamma\big(j_{p_c}^*R\Phi\QCe\big)\simeq \Coker\big(\Q_\ell\xrightarrow{\Delta} \Q_\ell\oplus \Q_\ell\big)(-1)&\longrightarrow \Q_\ell(-1)\\
(x_0,x_1)&\longmapsto x_1-x_0,
\end{align*}
where the component $x_0$ corresponds to the irreducible component $\fC^{v_0}_s$.
Let $r_C$ be the restriction map
\[R^1\Gamma\big(j_C^*R\Phi\QCe\big)\rightarrow R^1\Gamma\big(j_{p_c}^*R\Phi\QCe\big) .\]

\begin{lem}\label{lem:composition}
The composition $s\circ r_C\circ \ff^*_\Phi$ of the following morphisms
\[ R^1\Gamma\big(j^*R\Phi\QXe\big) \xrightarrow{\ff^*_\Phi}R^1\Gamma\big(j_C^*R\Phi\QCe\big)\xrightarrow{r_C}R^1\Gamma\big(j_{p_c}^*R\Phi\QCe\big)\xrightarrow{s} \Q_\ell(-1)\]
is equal to the map
\[w_{(v_0,e)}^* \colon R^1\Gamma\big(j^*R\Phi\QXe\big)\rightarrow\Q_\ell (-1).\]
\end{lem}
\begin{proof}
Let $p=\ff_s(p_c)\in \fXs$, let $j_p\colon p\times \widetilde{k^s}\hookrightarrow\fXbs$ be the closed immersion, and let \[r\colon R^1\Gamma\big(j^*R\Phi\QXe\big)\rightarrow R^1\Gamma\big(j_p^*R\Phi\QXe\big)\] be the restriction map.
We have a commutative diagram
\begin{equation}\label{eq:cd_vanishing2}
\begin{tikzcd}
R^1\Gamma\big(j^*R\Phi\QXe\big) \arrow{r}{\ff^*_\Phi} \arrow{d}{r} 
& R^1\Gamma\big(j_C^*R\Phi\QCe\big) \arrow{d}{r_C}\\
R^1\Gamma\big(j_p^*R\Phi\QXe\big) \arrow{r}{\ff^*_{\Phi,p_c}} 
& R^1\Gamma\big(j_{p_c}^*R\Phi\QCe\big).
\end{tikzcd}
\end{equation}
It follows from the cohomological interpretation of tropical weight vectors in \cite[Lemma 5.8]{Yu_Balancing_2013} and \cite[Corollary 3.5]{Berkovich_Vanishing_II_1996} that the composition $s\circ\ff^*_{\Phi,p_c}\circ r$ is equal to the map $w_{(v_0,e)}^*$.
We conclude our lemma by the commutativity of the diagram in \eqref{eq:cd_vanishing2}.
\end{proof}

Now let
\[
\sigma_{v_0} = \sum_{e \ni v_0} w_{(v_0,e)} \in \Ker\big(\Z^J\xrightarrow{\Sigma}\Z\big),
\]
summing over all edges $e$ of $\Gamma_0$ connected to $v_0$.
Using the isomorphism in \eqref{eq:vanishing_cycles_at_DIv}, $\sigma_{v_0}$ induces a linear map by duality
\[\sigma_{v_0}^*\colon R^1\Gamma\big(j^*R\Phi\QXe\big) \rightarrow \Q_\ell(-1).\]

\begin{lem}\label{lem:weight_vectors_by_models}
The map $\sigma_{v_0}^*\colon R^1\Gamma\big(j^*R\Phi\QXe\big)\rightarrow\Q_\ell(-1)$ is equal to the composition $\alpha^*_C\circ\ff^*_\Phi$ in \eqref{eq:cd_vanishing1}.
\end{lem}
\begin{proof}
By Theorem \ref{thm:vanishing_cycles}, the map $\alpha^*_C$ in \eqref{eq:cd_vanishing1} is induced by the cycle class map. So $\alpha^*_C$ is the sum of the maps $s\circ r_C$ that we considered in Lemma \ref{lem:composition} over all edges $e$ connected to $v_0$.
On the other hand the map \[\sigma_{v_0}^*\colon R^1\Gamma\big(j^*R\Phi\QXe\big)\rightarrow\Q_\ell(-1)\] is the sum of the maps
\[w_{(v_0,e)}^*\colon R^1\Gamma\big(j^*R\Phi\QXe\big)\rightarrow \Q_\ell(-1) \]
over all edges $e$ connected to $v_0$.
Therefore Lemma \ref{lem:weight_vectors_by_models} follows from Lemma \ref{lem:composition}.
\end{proof}

Combining Lemma \ref{lem:weight_vectors_by_models} and the commutativity of the diagram in \eqref{eq:cd_vanishing1}, we obtain the following lemma.

\begin{lem}\label{lem:weight_vectors_by_models_2}
The map $\sigma_{v_0}^*\colon R^1\Gamma\big(j^*R\Phi\QXe\big)\rightarrow\Q_\ell(-1)$ is equal to the composition of the following morphisms
\[ R^1\Gamma\big(j^*R\Phi\QXe\big)\xrightarrow{\alpha^*} H^2\et\big(j^*\QXbs\big)\xrightarrow{\ff_s^*}H^2\et\big(\fC^{v_0}_{\bar s},\Q_\ell\big)\xrightarrow{\sim} \Q_\ell(-1).\]
\end{lem}

\cref{lem:weight_vectors_by_models_2} can be reformulated as follows using \cref{thm:vanishing_cycles} and duality.

\begin{lem}\label{lem:tropical_weight_vector_as_intersection_number}
For every $i\in J=J_{I_v}$, the $i^\text{th}$ component $\sigma_{v_0}^i$ of $\sigma_{v_0}\in\Ker\big(\Z^J\xrightarrow{\Sigma}\Z\big)$ is equal to the degree of the line bundle $\big({\ff_s}|_{\fC_s^{v_0}}\big)^*\mathcal O(D_i)$ on the curve $\fC_s^{v_0}$.\qed
\end{lem}

Now let $Z\in Z_1^+(D_{I_v})$ be the pushforward $\ff_{s*}([\fC_s^{v_0}])$ of the fundamental class of $\fC^{v_0}_s$.
\cref{lem:tropical_weight_vector_as_intersection_number} implies that for every $i\in\IX$, the $i^\text{th}$ component $\sigma_{v_0}^i$ of $\sigma_{v_0}$ is equal to the intersection number $Z\cdot\cO(D_i)|_{D_{I_v}}$.
Therefore, $\sigma_{v_0}$ lies in the image of the map $\alpha$ in \eqref{eq:map_alpha}.
Since $\sigma_v$ equals the sum of $\sigma_{v_0}$ over all vertices $v_0$ of $\Gamma_0$ that maps to $v$, we conclude that $\sigma_v$ lies in the image of the map $\alpha$, completing the proof of \cref{thm:balancing_conditions}.

\section{Kähler structures}\label{sec:Kahler}

As we will study tropical curves and tropicalization of curves in families, we need an extra structure to ensure that the moduli spaces we encounter will be of finite type.
In this section, we introduce the notion of simple density and explain its relation with non-archimedean Kähler structures.

We use the settings of \cref{sec:basic_settings_(tropicalization_moduli)}.

\begin{defin}\label{def:simple_Kahler_str}
A \emph{simple density} $\omega$ on the Clemens polytope $\SX$ is a collection of numbers $\omega_{I,j}\in (0,+\infty]$ for every face $\Delta^I, I\subset\IX$ and every vertex $j\in J_I$, such that $\omega_{I,j}\ge\omega_{I',j}$ whenever $I\supset I'$ and $j\in J_I$.
\end{defin}

\begin{defin}\label{def:tropical_degree}
Let $(\Gamma,h)$ be a parametrized tropical curve in the Clemens polytope $\SX$ equipped with a simple density $\omega$.
Let $v$ be a vertex of $\Gamma$ such that $h(v)$ sits in the relative interior of a face $\Delta^{I_v}$ for some $I_v\subset \IX$.
The \emph{local degree} of $(\Gamma,h)$ at $v$ with respect to $\omega$ is by definition the real number
\[\abs{\sigma_v}_\omega \coloneqq\max_{j\in J_{I_v}} \omega_{I_v,j}\cdot \abs{\sigma_v^j},\]
where $\sigma_v$ denotes the sum of weight vectors around $v$.
The \emph{tropical degree} of the parametrized tropical curve $(\Gamma,h)$ with respect to the simple density $\omega$ is the sum of the local degrees over all vertices of $\Gamma$.
The marked points on a parametrized tropical curve do not contribute to the degree.
\end{defin}

\begin{rem}
The definition of tropical degree with respect to a simple density $\omega$ carries over to combinatorial types in $\SX$.
The tropical degree of a parametrized tropical curve in $\SX$ with respect to $\omega$ coincides with the tropical degree of its associated combinatorial type with respect to $\omega$.
\end{rem}

The notion of simple density is supposed to be an approximation to the non-archimedean Kähler structure introduced in \cite{Yu_Gromov_2014}\footnote{It is based on the work of Kontsevich and Tschinkel \cite{Kontsevich_Non-archimedean_2002}, see also \cite{Boucksom_Singular_2011}.}.
We recall that a Kähler structure $\hL$ on the \kanal space $X$ with respect to the formal model $\fX$ is a virtual line bundle $L$ on $X$ with respect to $\fX$ equipped with a strictly convex metrization $\hL$.

A Kähler structure $\hL$ on $X$ with respect to $\fX$ induces a simple density $\omega$ on the Clemens polytope $\SX$ in the following way.

For every $i\in\IX$, let $N^1(D_i)$ denote the numerical classes of divisors in $D_i$.
Put $N^1(D_i)_\R=N^1(D_i)\otimes\R$.
The curvature of the Kähler structure $\hL$ is a collection of ample classes $\partial_i \varphi_i \in N^1(D_i)_\R$ satisfying the following compatibility condition: for any $I\subset\IX$, any two vertices $i,j\in I$, we have
\begin{equation}\label{eq:compatibility_of_curvature}
(\partial_i \varphi_i)|_{D_I} = (\partial_j \varphi_j)|_{D_I}.
\end{equation}

Let $\Delta^I$ be a face of $\SX$ for some $I\subset \IX$ and let $i\in I$ be a vertex.
Let $\partial_I\varphi_i$ be the restriction of $\partial_i\varphi_i$ to the stratum $D_I$.
It is an ample class in $N^1(D_I)_\R$.
Let $j$ be an element in $J_I$, and let
\[\omega_{I,j} = \min\Set{C\cdot(\partial_I\varphi_i) | C\in\overline\NE(D_I), |C\cdot \mathcal O(D_j)|_{D_I}|=1},\]
where $\overline\NE(D_I)$ denotes the closure of the cone of effective proper curves in the stratum $D_I$.
Equation \eqref{eq:compatibility_of_curvature} shows that $\omega_{I,j}$ does not depend on the choice of $i\in I$.

The ampleness of the class $\partial_I \varphi_i$ in $N^1(D_I)_\R$ implies that $\omega_{I,j}$ is a positive real number. 
The fact that $\omega_{I,j}\ge\omega_{I',j}$ for any $I\supset I'$ follows from the inclusion $\overline\NE\big(D_I\big)\subset\overline\NE\big(D_{I'}\big)$.
So the collection
\[\omega=\Set{\omega_{I,j} | I\subset\IX, j\in J_I}\]
is a simple density on the Clemens polytope $\SX$ (Definition \ref{def:simple_Kahler_str}).

\begin{defin}\label{def:induced_simple_density}
The simple density $\omega$ on $\SX$ constructed above is called the simple density induced by the Kähler structure $\widehat L$ on $X$.
\end{defin}

Now let $C$ be a connected proper smooth \kanal curve and $f\colon C\to X$ a morphism.
The \emph{degree} of the morphism $f$ with respect to the Kähler structure $\hL$ is by definition the degree of the virtual line bundle $f^*L$ on the curve $C$.

Let $(\Gamma,h)$ denote the associated parametrized tropical curve in $\SX$.
We can relate the degree of the morphism $f$ to the tropical degree (Definition \ref{def:tropical_degree}) of $(\Gamma,(\gamma_i),h)$ as follows.

\begin{prop}\label{prop:analytic-tropical_degree}
The degree of the morphism $f\colon C\rightarrow X$ with respect to the Kähler structure $\widehat L$ is greater than or equal to the tropical degree of the associated parametrized tropical curve $(\Gamma,h)$ with respect to the simple density $\omega$ on the Clemens polytope $\SX$ induced by $\widehat L$.
\end{prop}
\begin{proof}
Let $\fC$, $\ff\colon\fC\to\fX$, $(\Gamma_0,h_0)$ be as in \cref{sec:parametrized_tropical_curves}.

Let $v_0$ be a vertex of $\Gamma_0$ and let $\fC_s^{v_0}$ denote the corresponding irreducible component of $\fC_s$.
Assume that $h_0(v_0)$ sits in the relative interior of a face $\Delta^{I_v}$ for some subset $I_v\subset\IX$.
Let $\sigma_{v_0}$ be the sum of weight vectors around $v_0$.

Let $i\in I_v$.
Let us call the degree of the line bundle $\big({\ff_s}|_{\fC_s^{v_0}}\big)^*(\partial_i\varphi_i)$ on the curve $\fC_s^{v_0}$ the \emph{local degree} of $\ff\colon \fC\rightarrow\fX$ at the irreducible component $\fC_s^{v_0}$. By the construction of the simple density $\omega$ and \cref{lem:tropical_weight_vector_as_intersection_number}, the local degree of $\ff\colon \fC\rightarrow\fX$ at the irreducible component $\fC_s^{v_0}$ is at least $\omega_{I_v,j}\cdot\abs{\sigma_{v_0}^j}$ for any $j\in J_{I_v}$.
Taking maximum over $j\in J_{I_v}$, we deduce that the local degree of $\ff\colon \fC\rightarrow\fX$ at the irreducible component $\fC_s^{v_0}$ is at least $\abs{\sigma_{v_0}}_\omega$ (see \cref{def:tropical_degree}).

Moreover, we observe that the tropical degree of $(\Gamma_0,h_0)$ is greater than or equal to the tropical degree of $(\Gamma,h)$.
Now the proposition follows from the fact that the degree of a virtual line bundle on a $k$-analytic curve equals the degree of its curvature (\cite[Proposition 5.7]{Yu_Gromov_2014}).
\end{proof}

\begin{rem}
We do not need the following fact in this paper, but it is worth pointing out.
We use the setting in the proof of \cref{prop:analytic-tropical_degree} and the terminology from \cite{Yu_Gromov_2014}.
The strictly convex metrization $\widehat L$ determines a germ of a strictly convex simple function $\varphi_i$ at the vertex $i$ up to addition by linear functions.
By the definition of the derivative $\partial_i\varphi_i$ and by \cref{lem:tropical_weight_vector_as_intersection_number}, we have
\[\deg\Big(\big(\ff_s|_{\fC_s^{v_0}}\big)^*(\partial_i\varphi_i)\Big) = \sum_{j\in J_{I_v}} \varphi_i(j)\cdot\deg\Big(\big(\ff_s|_{\fC_s^{v_0}}\big)^*\mathcal O(D_j)\Big) = \sum_{j\in J_{I_v}}\varphi_i(j)\cdot\sigma_{v_0}^j.\]
As a result, the degree of the morphism $f\colon C\to X$ with respect to the Kähler structure $\hL$ can be read out from the corresponding tropical curve.
\end{rem}

\section{The moduli space of tropical curves}\label{sec:space_of_tropical_curves}

In this section, we study the structure of the space of tropical curves with bounded degree.
The goal is to prove the following theorem.

\begin{thm}\label{thm:tropical_compactness}
Let $\SX$ be a Clemens polytope equipped with a simple density $\omega$.
Fix two non-negative integers $g,n$ and a positive real number $A$.
Denote by $M_{g,n}(\SX,A)$ the set of simple $n$-pointed genus $g$ parametrized tropical curves in $\SX$ whose tropical degree with respect to $\omega$ is bounded by $A$.
Then $M_{g,n}(\SX,A)$ is naturally a compact topological space with a stratification whose open strata are open convex polyhedrons.
The strata are in one-to-one correspondence with good simple $n$-pointed genus $g$ combinatorial types in $\SX$ whose tropical degree is bounded by $A$.
\end{thm}

\begin{defin}\label{def:weight_vectors}
For any $w\in \Z^{\IX}$, we define its \emph{norm} $|w|\coloneqq\sqrt{\sum(w^i)^2}$.
\end{defin}

\begin{lem}\label{lem:bound_type_A}
There exists an integer $N$, such that for any $n$-pointed parametrized tropical curve $(\Gamma,(\gamma_i),h)\in M(\SX,A)$, the number of type A vertices of $\Gamma$ and the norms of the sums $\sigma_v$ of tropical weight vectors around all type A vertices are bounded by $N$.
\end{lem}
\begin{proof}
It follows from Definitions \ref{def:type_of_vertices} and \ref{def:tropical_degree}.
\end{proof}

\begin{lem}\label{lem:bound_vertices}
There exists an integer $N$, such that for any $n$-pointed parametrized tropical curve $(\Gamma,(\gamma_i),h)\in M_{g,n}(\SX,A)$, the number of vertices of $\Gamma$ is bounded by $N$.
\end{lem}
\begin{proof}
Let $(\Gamma,(\gamma_i),h)\in M_{g,n}(\SX,A)$.
By computing the Euler characteristic of $\Gamma$, we have
\begin{equation}\label{eq:bound_vertices}
\# V(\Gamma) - \# E(\Gamma) = 1 - b_1(\Gamma) \ge 1-g.
\end{equation}
Since $(\Gamma,(\gamma_i),h)$ is simple, by \cref{lem:bound_type_A}, the number of vertices of $\Gamma$ of degree less than 3 can be bounded uniformly by an integer $N_0$.
Therefore, we have
\[ \# E(\Gamma) \ge \frac{3}{2} (\# V(\Gamma) - N_0).\]
Combining with \eqref{eq:bound_vertices}, we deduce that
\[\# V(\Gamma)\le 3 N_0+2g-2.\]
\end{proof}

\begin{lem}\label{lem:bound_edges}
There exists an integer $N$, such that for any $n$-pointed parametrized tropical curve $(\Gamma,(\gamma_i),h)\in M_{g,n}(\SX,A)$, the number of edges of $\Gamma$ is bounded by $N$.
\end{lem}
\begin{proof}
It follows from \cref{lem:bound_vertices} and the inequality \eqref{eq:bound_vertices}.
\end{proof}

\begin{defin}\label{def:path}
Let $(\Gamma,(\gamma_i),h)$ be an $n$-pointed parametrized tropical curve in $\SX$.
Let $(v_0,e_0)$ be a flag of $\Gamma$.
Let $i\in\IX$.
A \emph{path} $P$ starting from the flag $(v_0,e_0)$ with direction $i$ is a sequence of flags
$\big((v_0,e_0),(v_1,e_1),\dots,(v_{l_P},e_{l_P})\big)$ satisfying the following conditions:
\begin{enumerate}[(i)]
\item The vertices $v_j$, $v_{j+1}$ are the two endpoints of $e_j$ for $0\le j < l_P$;
\item We have $w^i_{(v_j,e_j)}>0$ for $0\le j\le l_P$;
\item Let $v'_{l_P}$ denote the endpoint of $e_{l_P}$ different from $v_{l_P}$, then $v'_{l_P}$ is a vertex of type A.
We say that the path $P$ ends at the vertex $v'_{l_P}$.
\end{enumerate}
\end{defin}

\begin{lem}\label{lem:bound_weight_vectors}
There exists an integer $N$, such that for any $n$-pointed parametrized tropical curve $(\Gamma,(\gamma_i),h)\in M_{g,n}(\SX,A)$, any flag $(v_0,e_0)$ of $\Gamma$, the norm of the weight vector $w_{(v_0,e_0)}$ is bounded by $N$.
\end{lem}
\begin{proof}
We fix $i\in \IX$.
Let $(\Gamma,(\gamma_i),h)\in M_{g,n}(\SX,A)$ and let $(v_0,e_0)$ be a flag of $\Gamma$.
Let $m\coloneqq w^i_{(v_0,e_0)}$ be the $i\th$ component of the weight vector $w_{(v_0,e_0)}$.
We assume that $m\ge 0$, otherwise we replace $v_0$ by the other endpoint of $e_0$.
We will show that $m$ can be bounded independently of $(\Gamma,(\gamma_i),h)$.

We use the ideas from \cite[\S3]{Yu_Number_2013}.
We claim that there exists a collection of $m$ paths $P_1,\dots,P_m$ starting from the flag $(v_0,e_0)$ with direction $i$ such that for any flag $(v,e)$ of $\Gamma$, the number of times that $(v,e)$ occurs in $P_1,\dots,P_m$ is bounded by $\abs{w^i_{(v,e)}}$.
Such a collection of $m$ paths can be constructed as follows.

We assign to each edge $e$ of $\Gamma$ an integer $c^i(e)$ called the \emph{capacity} (in the $i^\text{th}$ direction).
For each edge $e$ of $\Gamma$, we set initially $c^i(e)=\abs{w^i_{(v,e)}}$, where $v$ is an endpoint of $e$.
By definition, for every vertex $v$ of type $B$, we have
\begin{equation}\label{eq:capacity}
\sum_{e\ni v,\ w^i_{(v,e)}<0} c^i(e) = \sum_{e\ni v,\ w^i_{(v,e)}>0} c^i(e),
\end{equation}
where the left sum is over every edge $e$ connected to $v$ with $w^i_{(v,e)}<0$, and the right sum is over every edge $e$ connected to $v$ with $w^i_{(v,e)}>0$.

To construct the path $P_1$, we start with the flag $(v_0,e_0)$, and we decrease the capacity $c^i(e_0)$ by 1.
Suppose we have constructed a sequence of flags $(v_0,e_0),\dots,(v_j,e_j)$ satisfying \cref{def:path} Conditions (i)-(ii).
Let $v'_j$ denote the endpoint of the edge $e_j$ different from $v_j$.
We choose $(v_{j+1},e_{j+1})$ to be a flag of $\Gamma$ such that
\begin{enumerate}[(i)]
\item $v_{j+1}=v'_j$,
\item $w^i_{(v_{j+1},e_{j+1})}>0$,
\item $c^i(e_{j+1})>0$.
\end{enumerate}
If such a flag does not exist, we stop.
We note that in this case $v'_j$ is necessarily a vertex of type A because of \cref{eq:capacity}.
If such a flag exists, we let $(v_{j+1},e_{j+1})$ be the next flag in the path $P_1$ and we decrease the capacity $c^i(e_{j+1})$ by 1.
We iterate until we stop.
As a result we obtain the path $P_1$.
Then we construct the path $P_2$ by the same procedure as the construction of $P_1$ but with respect to the decreased capacities $c^i(e)$.
Since the path $P_1$ ends at a vertex of type $A$, we note that after the construction of $P_1$, the decreased capacities $c^i(e)$ still satisfy \cref{eq:capacity} for every vertex of type B except possibly for the vertex $v_0$.
Therefore, the path $P_2$ must also end at a vertex of type $A$.
We iterate this procedure and we obtain a collection of $m$ paths $P_1,\dots,P_m$.

By construction, for any vertex $v$ of $\Gamma$ of type A, the number of paths ending on $v$ is bounded by $\abs{\sigma_v^i}$.
In other words, except for at most $\abs{\sigma^i_v}$ paths, the other paths starting from $(v_0,e_0)$ with direction $i$ that reach the vertex $v$ must continue going.
Therefore, by \cref{lem:bound_type_A}, the number $m$ can be bounded independently of $(\Gamma,(\gamma_i),h)$.
\end{proof}

\begin{rem}
\cref{prop:analytic-tropical_degree} combined with Lemmas \ref{lem:bound_vertices}, \ref{lem:bound_edges} and \ref{lem:bound_weight_vectors} controls the complexity of the tropical curves obtained from analytic curves in $X$ with bounded degrees.
More precisely, given a positive real number $A$, there exists an integer $N$, such that for any connected proper smooth $k$-analytic curve $C$, any morphism $f:C\rightarrow X$ with degree bounded by $A$, the number of vertices, the number of edges and the norms of the weight vectors of the edges of the associated parametrized tropical curve in $\SX$ are all bounded by $N$.
We remark that this finiteness property can also be obtained as a consequence of Lemma \ref{lem:tropical_weight_vector_as_intersection_number} and the boundedness of the moduli stack of formal stable maps (\cite[Corollary 9.5]{Yu_Gromov_2014}).
\end{rem}

\begin{prop}\label{prop:space_fixed_combinatorial_type}
Let $\alpha\coloneqq(\Gamma,(\gamma_i))$ be a combinatorial type occurring in $M_{g,n}(\SX,A)$.
Let $M_\alpha$ denote the subset of $M_{g,n}(\SX,A)$ with combinatorial type $\alpha$.
Then $M_\alpha$ is naturally an open convex polyhedron in a real affine space, i.e.\ a subset 
of a real affine space given by the conjunction of finitely many linear strict inequalities.
Moreover, let $\oM_\alpha$ denote its closure in the real affine space.
We have a natural map
\[\iota_\alpha\colon\oM_\alpha \longrightarrow M_{g,n}(\SX,A).\]
The image of $\iota_\alpha$ coincides with the union
\[\bigcup_{\alpha'\preceq\alpha} M_{\alpha'}\]
in $M_{g,n}(\SX,A)$, where the union is over every combinatorial type $\alpha'$ which is the simplification of a degeneration of $\alpha$.
\end{prop}
\begin{proof}
Our proposition is related to \cite[Proposition 2.23]{Mikhalkin_Enumerative_2005}, \cite[Lemma 2.2]{Shustin_A_tropical_approach_2005} and \cite[Propositions 3.9, 3.12]{Gathmann_The_numbers_of_tropical_plane_curves_2007}.

An $n$-pointed parametrized tropical curve $(\Gamma,(\gamma_i),h)$ in $\SX$ with the given combinatorial type $(\Gamma,(\gamma_i))$ is determined by the position $h(v)$ of every vertex $v$ of $\Gamma$.
Since $\Gamma$ is connected and the tropical weight vectors are given by the combinatorial type, if we fix a ``root vertex'' $v_0$ of $\Gamma$, the set $M_\alpha$ is naturally a subset of the real affine space $\R^{\IX}\times \R^{E(\Gamma)}$ whose coordinates are
\begin{itemize}
\item the position $h(v_0)\in\SX\subset\R^{\IX}$ of the fixed root vertex $v_0$,
\item the lengths of the image $h(e)\subset\SX$ for every edge $e$ of $\Gamma$.
\end{itemize}
It is cut out by finitely many linear equations and linear strict inequalities corresponding to the following conditions:
\begin{enumerate}[(i)]
\item For every vertex $v$ of $\Gamma$, the image $h(v)$ lies in the relative interior of the face $\Delta^{I_v}$ of $\SX$.
\item The lengths of the images $h(e)$ are positive.
\item For every cycle in the graph $\Gamma$, the image of this cycle closes up in $\SX$.
\end{enumerate}
Note that a different choice of the root vertex $v_0$ corresponds to an affine automorphism.
So we have proved the first part of our proposition.
To describe the closure $\oM_\alpha$, it suffices to change the Conditions (i)-(ii) above to the following:
\begin{enumerate}[(i')]
\item For every vertex $v$ of $\Gamma$, the image $h(v)$ lies in the face $\Delta^{I_v}$ of $\SX$.
\item The lengths of the images $h(e)$ are non-negative.
\end{enumerate}
So strict inequalities change to non-strict inequalities and we obtain all possible degenerations of the combinatorial type $(\Gamma,(\gamma_i))$.
\end{proof}

\begin{proof}[Proof of \cref{thm:tropical_compactness}]
Lemmas \ref{lem:bound_vertices}, \ref{lem:bound_edges} and \ref{lem:bound_weight_vectors} imply that there are only finitely many combinatorial types occurring in the set $M_{g,n}(\SX,A)$.
Now we can construct $M_{g,n}(\SX,A)$ by attaching the closed polyhedrons of the form $\oM_\alpha$ in \cref{prop:space_fixed_combinatorial_type} according to the partial order given by degenerations of combinatorial types.
We conclude our proof of the theorem by \cref{prop:space_fixed_combinatorial_type}.
\end{proof}

\section{The moduli stack of non-archimedean analytic stable maps}\label{sec:stack_of_stable_maps}

In this section, we consider families of \kanal curves in our \kanal space $X$.
In order to compactify the universal family of such curves, we introduce the analog of Kontsevich's stable map in non-archimedean analytic geometry.

\begin{defin}[cf.\ {\cite[\S 1.1]{Kontsevich_Enumeration_1995}}, {\cite[\S 2]{Abramovich_Stable_2001}}]\label{def:stable_maps_(tropicalization_moduli)}
Let $T$ be a \kanal space.
An \emph{\gn \kanal stable map} $\big(C\rightarrow T, (s_i), f\big)$ into $X$ over $T$ consists of a morphism $C\rightarrow T$, a morphism $f\colon C\rightarrow X$ and $n$ morphisms $s_i\colon T\rightarrow C$ such that
\begin{enumerate}[(i)]
\item The morphism $C\rightarrow T$ is a proper flat\footnote{We refer to \cite{Ducros_Families_2011,Abbes_Elements_2010,Bosch_Formal_I,Bosch_Formal_II} for the notion of flatness in non-archimedean analytic geometry.} family of curves;
\item The geometric fibers of $C\rightarrow T$ are reduced with at worst double points as singularities, and are of arithmetic genus $g$;
\item The $n$ morphisms $s_i\colon T\rightarrow C$ are disjoint sections of $C\rightarrow T$ which land in the smooth locus of $C\rightarrow T$;
\item (Stability condition) For any geometric point $t$ in $T$, the automorphism group of the fiber $\big(C_t\to\{t\}, (s_i(t)), f_t\colon C_t\to X\big)$ is a finite \kanal group\footnote{It is a finite constant group when the field $k$ has characteristic zero.}.
\end{enumerate}
\end{defin}

Let $\hL$ be a Kähler structure on $X$ with respect to a strictly semi-stable formal model $\fX$ of $X$.
Fix a positive real number $A$.
Let $\bcMgn(X,A)$ denote the moduli stack of \gn \kanal stable maps into $X$ whose degree with respect to $\hL$ is bounded by $A$.

The main result of \cite{Yu_Gromov_2014} is the following theorem.

\begin{thm}[Non-archimedean Gromov compactness]\label{thm:non-archimedean_Gromov_(tropicalization_moduli)}
The stack $\bcMgn(X,A)$ is a compact \kanal stack.
If we assume moreover that the \kanal space $X$ is proper and that the residue field $\tilde k$ has characteristic zero, then $\bcMgn(X,A)$ is a proper \kanal stack.
\end{thm}

Roughly, the theorem above means that the abstract moduli stack $\bcMgn(X,A)$ locally looks like \kanal spaces.
We will study the tropicalization of the \kanal moduli stack $\bcMgn(X,A)$ in \cref{sec:continuity_of_tropicalization,sec:polyhedrality}.

\section{Continuity of tropicalization}\label{sec:continuity_of_tropicalization}

Every $k$-analytic stable map into the $k$-analytic space $X$ gives rise to a parametrized tropical curve in the Clemens polytope $\SX$.
So we obtain a map from the moduli space of $k$-analytic stable maps to the moduli space of tropical curves.
We call this map the tropicalization map of the moduli space of $k$-analytic stable maps.
In Section \ref{sec:continuity_of_tropicalization}, we prove the continuity of this tropicalization map using the balancing conditions in \cref{sec:balancing_conditions} and the formal models of families of stable maps developed in \cite{Yu_Gromov_2014}.
In Section \ref{sec:polyhedrality}, we prove that the image of this tropicalization map is compact and polyhedral using the continuity theorem together with the quantifier elimination theorem from the model theory of rigid subanalytic sets.

Let $X$ be a $k$-analytic space, $\fX$ a strictly semi-stable formal model of $X$, and $\hL$ a Kähler structure on $X$ with respect to $\fX$.
Let $\omega$ be the simple density on the Clemens polytope $\SX$ induced by the Kähler structure $\widehat L$ (Definition \ref{def:induced_simple_density}).

Fix two non-negative integers $g,n$ and a positive real number $A$.
Let $M_{g,n}(\SX,A)$ denote the space of simple $n$-pointed genus $g$ parametrized tropical curves in $\SX$ whose degree with respect to $\omega$ is bounded by $A$ (see \cref{thm:tropical_compactness}).

Now let $T$ be a strictly $k$-analytic space and let $\big(C\rightarrow T,(s_i),f\big)$ be an \gn $k$-analytic stable map into $X$ over $T$.
Assume that the degree of every geometric fiber $\big(C_t\to\{t\},(s_i(t)),f_t\colon C_t\to X\big)$ with respect to $\hL$ is bounded by $A$.
By \cref{prop:analytic-tropical_degree}, we obtain a set-theoretic tropicalization map $\tau_T$ from $T$ to $M_{g,n}(\SX,A)$.

\begin{thm}\label{thm:continuity}
The map $\tau_T\colon T\rightarrow M_{g,n}(\SX,A)$ is continuous.
\end{thm}

We begin the proof by introducing some new notions.

\begin{defin}
Let $\nu$ be a positive integer.
A \emph{$\nu$-semi-stable subdivision} $\oSX$ of the Clemens polytope $\SX$ is a finite rational simplicial subdivision of $\SX$ such that the vertices of $\oSX$ lie in the lattice $\frac{1}{\nu}\Z^{\IX}\subset\R^{\IX}$ and that every $d$-dimensional simplex of $\oSX$ has volume $\sqrt{d+1}/(\nu^d d!)$.
\end{defin}

Let $\oSX$ be a $\nu$-semi-stable subdivision of $\SX$.
Let $\oIX$ denote the set of vertices of $\oSX$.
Let $\Delta^{\oIX}$ denote the simplex in $\R^{\oIX}_{\ge 0}$ given by the equation $\sum x_i = 1/\nu$.
We have natural embeddings $\oSX\subset\Delta^{\oIX}\subset\R^{\oIX}$.
So we can consider parametrized tropical curves and combinatorial types in $\oSX$ as in Definitions \ref{def:parametrized_tropical_curve} and \ref{def:combinatorial_type}.
The map $\oSX\to\SX$ induces a linear map of vector spaces $\R^{\oIX}\to\R^{\IX}$ and a morphism of lattices $\Z^{\oIX}\to\Z^{\IX}$.
In this way, parametrized tropical curves in $\oSX$ project to parametrized tropical curves in $\SX$.

\begin{defin}
A \emph{subdivision datum} $\cD=\big(\oSX,(\Gamma,(\gamma_i))\big)$ consists of
\begin{enumerate}[(i)]
\item a $\nu$-semi-stable subdivision $\oSX$ of the Clemens polytope $\SX$ for a positive integer $\nu$, and
\item an $n$-pointed combinatorial type $(\Gamma,(\gamma_i))$ in $\oSX$.
\end{enumerate}
\end{defin}

Given a subdivision datum $\cD=\big(\oSX,(\Gamma,(\gamma_i))\big)$, let $M_{g,n}(\oSX,A)$ be the space of simple $n$-pointed genus $g$ parametrized tropical curves in $\oSX$ that project to parametrized tropical curves in $\SX$ with tropical degree bounded by $A$.
Via projection followed by simplification, we obtain a natural bijection $M_{g,n}(\oSX,A)\xrightarrow{\sim} M_{g,n}(\SX,A)$.

The space $M_{g,n}(\oSX,A)$ has a natural stratification whose open strata corresponds to combinatorial types in $\oSX$.
So it is a subdivision of the stratification of $M_{g,n}(\SX,A)$.

Let $\Xi(\cD)$ be the finite set of $n$-pointed combinatorial types $(\Gamma',(\gamma'_i))$ in $\oSX$ such that
\begin{enumerate}[(i)]
\item $(\Gamma',(\gamma'_i))$ occurs in the space $M_{g,n}(\oSX,A)$,
\item $(\Gamma,(\gamma_i))$ is the simplification of a degeneration of $(\Gamma',(\gamma'_i))$.
\end{enumerate}

For each $n$-pointed combinatorial type $(\Gamma',(\gamma'_i))\in\Xi(\cD)$, let $\Delta^\circ_{(\Gamma',(\gamma'_i))}$ denote the open stratum of $M_{g,n}(\oSX,A)$ corresponding to $(\Gamma',(\gamma'_i))$.
Put
\[U(\cD)\coloneqq\bigcup_{(\Gamma',(\gamma'_i))\in\Xi(\cD)} \Delta^\circ_{(\Gamma',(\gamma'_i))}.\]
By the construction of the topology on the space $M_{g,n}(\SX,A)$ in the proof of \cref{thm:tropical_compactness}, we see that $U(\cD)$ is an open subset of $M_{g,n}(\SX,A)$.
Moreover, by the construction of the polyhedral structure on every stratum of $M_{g,n}(\SX,A)$ in the proof of \cref{prop:space_fixed_combinatorial_type}, we observe that the subdivision $M_{g,n}(\oSX,A)$ of $M_{g,n}(\SX,A)$ can be as fine as possible when we refine the subdivision $\oSX$ of $\SX$.
So we have proved the following lemma.

\begin{lem}\label{lem:base_of_topology}
The subsets $U(\cD)$ for all subdivision data $\cD=\big(\oSX,(\Gamma,(\gamma_i))\big)$ form a base for the topology on $M_{g,n}(\SX,A)$.
\end{lem}

The next lemma relates degenerations of combinatorial types with formal models.

\begin{lem}\label{lem:degenerations_formal_model}
Let $\fT$ be a formal scheme of finite presentation over $\kc$.
Let $\big(\fC\to\fT,(\fs_i),\ff\big)$ be an \gn formal stable map into $\fX$  over $\fT$ (cf.\ \cite[\S 8]{Yu_Gromov_2014}).
Let $\pi_\fT\colon\fT_\eta\to\fT_s$ denote the reduction map.
Let $\bar b, \bar b'$ be two points in $\fT_s$ such that $\bar b'$ is a specialization of $\bar b$.
Let $\big(C_b,(s_i),f_b\big)$ (resp.\ $\big(C_{b'},(s'_i),f_{b'}\big)$) be an analytic stable map into $X$ corresponding to a point in $\pi\inv_\fT(\bar b)$ (resp.\ $\pi\inv_\fT(\bar b')$).
Let $(\Gamma_b,(\gamma_i),h_b)$ and $(\Gamma_{b'},(\gamma'_i),h_{b'})$ be the associated parametrized tropical curves in $\SX$ respectively.
Let $(\Gamma_b,(\gamma_i))$ and $(\Gamma_{b'},(\gamma'_i))$ be the associated combinatorial types in $\SX$ respectively.
Then $(\Gamma_b,(\gamma_i))$ is a degeneration of $(\Gamma_{b'},(\gamma'_i))$.
\end{lem}
\begin{proof}
Let $(C_{\bar b},(\bar s_i),f_{\bar b})$ (resp.\ $(C_{\bar b'},(\bar s'_i),f_{\bar b'})$) be the algebraic stable map into $\fXs$ corresponding to the point $\bar b$ (resp.\ $\bar b'$).
Then the nodal curve $C_{\bar b'}$ is a degeneration of the nodal curve $C_{\bar b}$.
In other words, $C_{\bar b}$ is a smoothing of $C_{\bar b'}$.
So we obtain a surjective map from the set of the irreducible components of $C_{\bar b'}$ to the set of the irreducible components of $C_{\bar b}$.
This induces a surjective map $\phi\colon V(\Gamma_{b'})\to V(\Gamma_b)$ satisfying \cref{def:degeneration_combinatorial_type} Conditions (\ref{item:degeneration:genus}) and (\ref{item:degeneration:points}).
It also induces bijections $\phi\colon E(u',v')\to E(\phi(u'),\phi(v'))$ for every pair of vertices $u'$, $v'$ of $\Gamma_{b'}$ such that $\phi(u')\neq\phi(v')$.
Since any node of $C_{\bar b}$ remains a node of $C_{\bar b'}$, any edge of $\Gamma_b$ is of the form $\phi(e')$ for some edge $e'$ of $\Gamma_{b'}$.
Therefore, \cref{def:degeneration_combinatorial_type} Condition (\ref{item:degeneration:edge}) is satisfied except possibly for the equalities of tropical weight vectors.

The assumption that $\bar b'$ is a specialization of $\bar b$ implies that if an irreducible component $C^v_{\bar b}$ of $C_{\bar b}$ maps to a closed stratum $D_I$ of $\fXs$, then the irreducible components of $C_{\bar b'}$ corresponding to the degeneration of $C^v_{\bar b}$ map to $D_I$ as well.
Therefore, \cref{def:degeneration_combinatorial_type} Condition (\ref{item:degeneration:face}) is satisfied.

Now let $\fT_s^b$ be the Zariski closure of the point $\bar b$.
Let $\fC^b_s\coloneqq\fC_s\times_{\fT_s}\fT^b_s$.
Let $\sigma_0$ be a node of $C_{\bar b}$.
The node $\sigma_0$ remains a node for the family of nodal curves $\fC^b_s$ over $\fT^b_s$.
So we obtain a section $\sigma\colon\fT^b_s\to\fC^b_s$.
Put $\fT^b_{\bar s}\coloneqq\fT^b_s\times\widetilde{k^s}$, $\fC^b_{\bar s}\coloneqq\fC^b_s\times\widetilde{k^s}$ and $\bar{\sigma}\coloneqq\sigma\times\widetilde{k^s}\colon\fT^b_{\bar s}\to\fC^b_{\bar s}$.

Assume that $\ff_s\colon\fC_s\to\fXs$ maps the section $\sigma$ into a closed stratum $D_I$ of $\fXs$.
Let $\oD_I\coloneqq D_I\times\widetilde{k^s}$ and let $j\colon\oD_I\hookrightarrow\fXbs$ denote the closed immersion as in \cref{sec:balancing_conditions}.
Since $\fC\to\fT$ is a family of nodal curves, the sheaf $\bar{\sigma}^* R^1\Phi\QCe$ is locally constant.
Let $\Lambda$ be the constant sheaf on $\fT^b_{\bar s}$ associated to the $\Q_\ell$-vector space $j^*R^1\Phi\QXe$.
The morphism $\ff\colon\fC\to\fX$ of formal schemes induces a morphism $\lambda\colon\Lambda\to\bar{\sigma}^* R^1\Phi\QCe$ of locally constant sheaves on $\fT^b_s$.
Assume that $\sigma(\bar b')$ corresponds to an edge $e'\in E(u',v')$ for two vertices $u', v'$ of $\Gamma_{b'}$.
Then $\sigma(\bar b)$ corresponds to the edge $\phi(e')$ of $\Gamma_b$.
By \cref{lem:composition}, the weight vector $w_{(u',e')}$  (resp.\ $w_{(\phi(u'),\phi(e'))}$) can be computed by the stalk of the morphism $\lambda$ at a geometric point over $\bar b'$ (resp.\ $\bar b$).
Since the morphism $\lambda$ is locally constant and the base $\fT^b_s$ is connected, we deduce that $w_{(u',e')}=w_{(\phi(u'),\phi(e'))}$.
In other words, the equalities of weight vectors in \cref{def:degeneration_combinatorial_type} Condition (\ref{item:degeneration:edge}) hold.
So we have proved our lemma.
\end{proof}

\begin{rem}
The part concerning vanishing cycles in the proof of \cref{lem:degenerations_formal_model} means intuitively that the winding numbers do not change under deformations.
\end{rem}

\begin{proof}[Proof of \cref{thm:continuity}]
In order to prove the continuity of the map $\tau_T\colon T\rightarrow M_{g,n}(\SX,A)$, it suffices to show that for any point $b\in T$ and any subdivision datum
$\cD=\big(\oSX,(\Gamma,(\gamma_i))\big)$
such that the associated open subset $U(\cD)$ contains the point $\tau_T (b)\in M_{g,n}(\SX,A)$, the inverse image $\tau_T^{-1} (U(\cD))$ is a neighborhood of $b$ in $T$.

We fix $b$ and $\cD$ as above.
Up to passing to a finite extension of the ground field $k$, we can find a strictly semi-stable formal model $\ofX$ for the $k$-analytic space $X$ such that the associated Clemens polytope $S_{\ofX}$ is isomorphic to $\oSX$ (cf.\ \cite{Kempf_Toroidal_1973}).
Using \cite[Theorem 1.5]{Yu_Gromov_2014}, replacing $T$ by a quasi-étale covering if necessary,
one can find a formal model $\ft\colon \fT\rightarrow\bcMgn(\ofX,A)$ for the morphism $t\colon T\rightarrow\bcMgn(X,A)$, i.e.\ an $n$-pointed genus $g$ formal stable map $\big(\fC\rightarrow\fT,(\fs_i),\ff\big)$ into $\ofX$ over $\fT$ which gives back the \kanal stable map $\big(C\to T,(s_i), f\big)$ into $X$ over $T$ when passing to generic fibers.
The finite ground field extension and the quasi-étale covering above are allowed in virtue of the descent of open immersions (cf.\ \cite[Theorem 4.2.7]{Conrad_Relative_2006}).

Let $\pi_\fT\colon T\simeq\fT_\eta\rightarrow\fT_s$ denote the reduction map.
Let $\bar b\coloneqq \pi_\fT(b)$ and let $\fT_s^b$ be the Zariski closure of the point $\bar b$.
By \cref{lem:degenerations_formal_model}, for any point $\bar b'\in\fT^b_s$ and any analytic stable map corresponding to a point in $\pi\inv_\fT(\fT_s^b)$, its associated tropical curve belongs to the subset $U(\cD)$.
In other words, we have
\[\pi^{-1}_\fT(\fT^b_s) \subset \tau_T^{-1}\big(U(\cD)\big).\]
By the anti-continuity of the reduction map, the set $\pi^{-1}_\fT(\fT^b_s)$ is open in $T$ for the Berkovich topology.
Since $b\in\pi^{-1}_\fT(\fT^b_s)$ by construction, we have proved that $\tau_T^{-1}\big(U(\cD)\big)$ is a neighborhood of $b$ in $T$.
Using Lemma \ref{lem:base_of_topology}, we conclude that the map $\tau_T\colon T\rightarrow M_{g,n}(\SX,A)$ is continuous.
\end{proof}

\begin{cor}\label{cor:continuity}
The tropicalization map $\tau_M\colon \bcMgn(X,A)\rightarrow M_{g,n}(\SX,A)$ is a continuous map.
\end{cor}

\begin{cor}\label{cor:compactness_of_Mtgn}
Let $\bMgnt(X,A)$ denote the image of the tropicalization map $\tau_M$.
Then $\bMgnt(X,A)$ is a compact subset of $M_{g,n}(\SX,A)$.
\end{cor}
\begin{proof}
It follows from Theorem \ref{thm:non-archimedean_Gromov_(tropicalization_moduli)} and Corollary \ref{cor:continuity}.
\end{proof}

\section{Polyhedrality via quantifier elimination}\label{sec:polyhedrality}

In Corollary \ref{cor:compactness_of_Mtgn}, we defined $\bMgnt(X,A)$ to be the subset of $M_{g,n}(\SX,A)$ consisting of tropical curves in $\SX$ that arise from \gn stable maps into $X$ with degree bounded by $A$.
In this section, we will show the polyhedral nature of this subset.

The proof uses quantifier elimination from model theory together with the continuity theorem of Section \ref{sec:continuity_of_tropicalization}.
The model-theoretic approach here is inspired by Antoine Ducros' work \cite{Ducros_Espaces_de_Berkovich_polytopes_2012}.
The model theory of algebraically closed valued fields is used by Ducros in \cite{Ducros_Espaces_de_Berkovich_polytopes_2012} after algebraization of $k$-analytic situations.
However, since we will deal not only with one single $k$-analytic curve but also with families of $k$-analytic curves, it is not clear how to apply the standard algebraization techniques.
Therefore we resort to the model theory of rigid subanalytic sets developed by Leonard Lipshitz and Zachary Robinson \cite{Lipshitz_Rigid_1993,Lipshitz_Model_completeness_2000}.

Let us take a quick review following \cite{Lipshitz_Rigid_1993,Lipshitz_Model_completeness_2000,Martin_Constructibilite_2013,Martin_Tameness_2015}.
We restrict to the case where the ground field $k$ is of discrete valuation because the theory in the general case is more involved.

The language $\LanD$ of the theory of rigid subanalytic sets consists of three sorts: $\cO$, $\fm$ and $\Gamma_0$.
There are binary function symbols $+,-,\cdot$ on $\cO$ and $\fm$, a relation symbol $<$ and a constant $0$ on $\Gamma_0$.
Moreover, there are functions symbols $D_0\colon\cO^2\to\cO$, $D_1\colon\cO^2\to\fm$, $|\cdot|\colon\cO\to\Gamma_0$,
and for each $f\in\kc\langle T_1,\dots,T_m \rangle \llb S_1,\dots,S_n\rrb$, there is a function symbol $f\colon\cO^m\times\fm^n\to\cO$.
If in addition $f$ is in the ideal $(\varpi,S_1,\dots,S_n)$, where $\varpi$ denotes a uniformizer of $k$, there is a function symbol $f\colon\cO^m\times\fm^n\to\fm$.
A term in the language $\LanD$ of the form $f\colon\cO^m\times\fm^n\to\cO$ will be called a \emph{D-function} for simplicity.

Given a non-archimedean field extension $k\subset K$, one can associate to $K$ a standard $\LanD$-structure.
We note that $\cO,\fm,\Gamma_0$ are interpreted as $K^\circ, K^{\circ\circ}$, $|K|$ respectively, and
$D_0$, $D_1$ are interpreted as
\begin{align*}
D_0(x,y)&=\begin{cases}
x/y &\text{if }\abs{x}\le\abs{y}\neq 0\\
0 &\text{otherwise,}
\end{cases}\\
D_1(x,y)&=\begin{cases}
x/y &\text{if }\abs{x}<\abs{y}\\
0 &\text{otherwise.}
\end{cases}\\
\end{align*}
The other symbols have obvious interpretations.

Let $\bbD$ (resp.\ $\bbD^\circ$) denote the closed (resp.\ open) unit disc over $k$, considered as a \kanal space.
Set $\Gamma\coloneqq\sqrt{\abs{k^*}}$.

\begin{defin}[cf.\ \cite{Martin_Constructibilite_2013}]\label{def:subanalytic}
A subset $S$ of $\bbD^m\times (\bbD^\circ)^n \times\Gamma^l$ is called \emph{subanalytic} if it is a boolean combination of subsets of the form
\begin{equation*}
\big\{(x,y,\gamma)\in \bbD^m\times (\bbD^\circ)^n \times\Gamma^l \ \big|\ \abs{f(x,y)}\gamma^u c\leq \abs{g(x,y)}\big\} ,
\end{equation*}
where $f,g$ are $D$-functions\footnote{To evaluate the norm of a $D$-function on a point of the $k$-analytic space $\bbD^m\times (\bbD^\circ)^n$, it suffices to pass to a ground field extension making the point rational.}, $u\in\Z^l$ and $c\in\Gamma$.
\end{defin}

\begin{rem}
Subanalytic sets are exactly those definable by first order formulas without quantifiers in the language $\LanD$.
\end{rem}

The following theorem has many variants in the literature \cite{Bieri_Geometry_1984,Berkovich_Smooth_II_2004,Einsiedler_Non-archimedean_2006,Gubler_Tropical_2007,Ducros_Espaces_de_Berkovich_polytopes_2012}.
The theory of rigid subanalytic sets provides us another proof via the quantifier elimination theorem of Lipshitz \cite[Theorem 3.8.2]{Lipshitz_Rigid_1993}.

\begin{thm}[cf.\ \cite{Martin_Dimensions_2014}]\label{thm:polyhedrality_of_tropicalization}
Let $Y\subset\bbD^m$ be a closed immersion and let $S\subset Y\subset\bbD^m$ be a subanalytic set as in Definition \ref{def:subanalytic}.
Let $h_1,\dots,h_l$ be $l$ $D$-functions on $\bbD^m$.
Then $(\val h_1,\dots,\val h_l)(S)\cap\R^l$ is a finite polyhedral complex in $\R^l$ of dimension less than or equal to the dimension of $Y$, where we put $\val 0 = +\infty$.
\end{thm}

We generalize Theorem \ref{thm:polyhedrality_of_tropicalization} to the relative case.
Let $Y\subset\bbD^{m'}$ and $T\subset\bbD^m$ be closed immersions.
Let $p_0\colon\bbD^{m'}\to\bbD^m$ be a morphism which restricts to a morphism $p\colon Y\to T$ of $k$-affinoid spaces.
Let $f$ be a morphism from $Y$ to the analytification $(\bbG_{\mathrm m,k}^n)\an$ of the algebraic torus $\bbG_{\mathrm m,k}^n$ over $k$.
Let $G$ be the definable set
\[G=\Big\{(t,\gamma)\in T\times \Gamma^l \ \Big|\  \exists y\in Y \Big(\big(p(y)=t\big)\wedge\big(|f|(y)=\gamma\big)\Big)\Big\}.\]
Using quantifier elimination for rigid subanalytic sets,
there exists a positive integer $N$, D-functions $q_1,\dots,q_N,q'_1,\dots,q'_N$ on $\bbD^m$,
$u_1,\dots,u_N\in\Z^l$ and $c_1,\dots,c_N\in\Gamma$ such that the set $G$ can be written as a boolean combination of the subsets
\[\big\{(t,\gamma)\in T\times\Gamma^l \ \big|\ \abs{q_i(t)}\gamma_i^{u_i} c_i\leq \abs{q'_i(t)}\big\} , \text{ for } i=1,\dots,N.\]
By Theorem \ref{thm:polyhedrality_of_tropicalization}, the intersection
\[(\val q_1,\dots,\val q_N,\val q'_1,\dots,\val q'_N)(T)\cap\R^{2N}\]
is a finite polyhedral complex in $\R^{2N}$ of dimension less than or equal to the dimension of $T$.

Let $M_{g,n}(\SX,A)_u$ denote the quotient of $M_{g,n}(\SX,A)$ by the following equivalence relation:
two parametrized tropical curves $(\Gamma,(\gamma_i),h)$ and $(\Gamma',(\gamma'_i),h')$ in $M_{g,n}(\SX,A)$ are in the same equivalence class if $h(\Gamma)=h'(\Gamma')$ as subsets of $\SX$ and $h(\gamma_i)=h'(\gamma'_i)$ for $1\le i\le n$.
As the space $M_{g,n}(\SX,A)$, the space $M_{g,n}(\SX,A)_u$ is also naturally a compact topological space with a stratification whose open strata are open convex polyhedrons.

\begin{prop}\label{prop:relative_polyhedrality_after_forget}
We use the settings in \cref{sec:continuity_of_tropicalization}.
Let $\tau_T\colon T\rightarrow M_{g,n}(\SX,A)$ denote the set-theoretically defined tropicalization map, and let $u$ denote the quotient map from $M_{g,n}(\SX,A)$ to $M_{g,n}(\SX,A)_u$.
Then the image of $T$ under the composite map $u\circ\tau_T$ is polyhedral in $M_{g,n}(\SX,A)_u$, in the sense that its intersection with every open stratum of $M_{g,n}(\SX,A)_u$ is polyhedral.
\end{prop}
\begin{proof}
We can assume that $T$ is a $k$-affinoid space.
Choose a finite covering of the formal model $\fX$ by affine open subschemes of the form $\fU$ as in Definition \ref{def:sss_formal_model}.
Let $\fU$ be an element in the covering and let $\{Y_j\}$ be a finite affinoid covering of the inverse image $f^{-1}(\fU_\eta)$.
Using the explicit description of the map $\tau\colon X\rightarrow\SX$ in \cref{rem:explicit}, we apply the model-theoretic arguments above to the morphisms $Y_j\rightarrow T$, and apply \cref{thm:polyhedrality_of_tropicalization} to the tropicalization of the marked points $s_i$.
We deduce that the image of $T$ under the composite map $u\circ\tau_T$ is polyhedral in $M_{g,n}(\SX,A)_u$.
\end{proof}

\begin{thm}\label{thm:relative_polyhedrality}
We use the setting of Proposition \ref{prop:relative_polyhedrality_after_forget}.
The image of $T$ under the map $\tau_T\colon T\rightarrow M_{g,n}(\SX,A)$ is polyhedral in $M_{g,n}(\SX,A)$, in the sense that its intersection with every open stratum of $M_{g,n}(\SX,A)$ is polyhedral.
\end{thm}
\begin{proof}
We can assume that $T$ is a $k$-affinoid space.
Let $n'$ be an integer greater than or equal to $n$.
Let $M_{g,n'}^v(\SX,A)$ denote the subset of $M_{g,n'}(\SX,A)$ consisting of parametrized tropical curves $(\Gamma,(\gamma_i),h)$ such that $n(v)\ge 1$ for every vertex $v$ of $\Gamma$.
Let $M_{g,n'}^v(\SX,A)_u$ denote the quotient of $M_{g,n'}^v(\SX,A)$ by the equivalence relation introduced before \cref{prop:relative_polyhedrality_after_forget}.
We note that both spaces $M_{g,n'}^v(\SX,A)$ and $M_{g,n'}^v(\SX,A)_u$ are naturally compact topological spaces with a stratification whose open strata are open convex polyhedrons.
Moreover, the fibers of the quotient map $u\colon M_{g,n'}^v(\SX,A)\to M_{g,n'}^v(\SX,A)_u$ are finite.

Up to passing to a finite quasi-étale covering of $T$, we can add enough new marked points so that the tropicalization map
\[\tau'_T\colon T\longrightarrow M_{g,n'}(\SX,A)\]
factorizes through the inclusion $M_{g,n'}^v(\SX,A)\subset M_{g,n'}(\SX,A)$ for some $n'\ge n$.
In order to show that the image $\tau_T(T)$ is polyhedral in $M_{g,n}(\SX,A)$, it suffices to show that the image $\tau'_T(T)$ is polyhedral in $M_{g,n'}^v(\SX,A)$.

Let $O$ be an open stratum of the space $M_{g,n'}^v(\SX,A)_u$.
Let $O_1,\dots,O_m$ denote the connected components of the inverse image $u^{-1}(O)$ in $M_{g,n'}^v(\SX,A)$.
We will show that the intersection $\tau'_T(T)\cap O_i$ is polyhedral for all $i=1,\dots,m$.
Let $T_O\subset T$ denote the inverse image $(u\circ\tau'_T)^{-1} (O)$.
By Theorem \ref{thm:continuity}, the restriction $(\tau'_T)|_{T_{O}} \colon T_{O}\rightarrow u^{-1}(O)$ is a continuous map.
Therefore, for any $i=1,\dots,m$, the inverse $T_{O_i}\coloneqq\big((\tau'_T)|_{T_{O}}\big)^{-1} (O_i)$ is a union of connected components of $T_{O}$.
The model-theoretic arguments above imply that the map $u\circ\tau'_T\colon T\rightarrow M_{g,n}(\SX,A)_u$ is given by the norms of a finite collection of $D$-functions $\{q_j\}$.
Therefore, the subspace $T_O\subset T$ is a subanalytic set.
Since the subspace $T_{O_i}$ is a union of connected components of $T_{O}$, it is also subanalytic (cf.\ \cite[\S 3]{Martin_Tameness_2015}).
Restricting the $D$-functions $\{q_j\}$ to $T_{O_i}$ and applying \cref{thm:polyhedrality_of_tropicalization}, we conclude that the image $\tau'_T(T_{O_i})=\tau'_T(T)\cap O_i$ is polyhedral, completing the proof.
\end{proof}

\begin{cor}\label{cor:polyhedrality}
Let $\bMgnt(X,A)$ denote the image of the tropicalization map
\[\tau_M\colon \bcMgn(X,A)\rightarrow M_{g,n}(\SX,A).\]
Then $\bMgnt(X,A)$ is a compact and polyhedral in $M_{g,n}(\SX,A)$.
\end{cor}
\begin{proof}
The corollary follows from Theorems \ref{thm:non-archimedean_Gromov_(tropicalization_moduli)}, \ref{thm:continuity} and \ref{thm:relative_polyhedrality}.
\end{proof}

\bibliographystyle{plain}
\bibliography{dahema}

\end{document}